\documentclass[11pt]{amsart}
\usepackage{amsmath}
\usepackage{bbding}
\usepackage{tikz}

\theoremstyle{plain}
\newtheorem{theorem}{Theorem}[section]
\newtheorem{lemma}[theorem]{Lemma}
\newtheorem{prop}[theorem]{Proposition}

\theoremstyle{definition}
\newtheorem{remark}[theorem]{Remark}

\newtheorem{example}[theorem]{Example}
\numberwithin{equation}{section}

\def\be{\begin{equation}}
\def\ee{\end{equation}}

\begin{document}

\title[On The Negativity of Ricci Curvatures]
{On The Negativity of Ricci Curvatures of\\ Complete Conformal Metrics}
\author{Qing Han$^1$, \, Weiming Shen$^2$ }

\begin{abstract}
A version of the singular Yamabe problem in bounded domains
yields complete conformal metrics
with negative constant scalar curvatures. In this paper,
we study whether these metrics have negative Ricci curvatures. Affirmatively, we prove that
these metrics indeed have negative Ricci curvatures in bounded convex domains in the
Euclidean space.
On the other hand, we provide
a general construction of domains in compact manifolds and demonstrate that
the negativity of Ricci curvatures does not hold if the boundary is close to certain sets
of low dimension. The expansion of the Green's function and the positive mass theorem play
essential roles in certain cases. \end{abstract}
\footnotetext[1]{Qing Han

qhan@nd.edu}
\footnotetext[2]{Weiming Shen \Envelope

wmshen@pku.edu.cn}
\footnotetext[1]{Department of Mathematics,
University of Notre Dame, Notre Dame, IN 46556, USA}
\footnotetext[2]{School of Mathematical Sciences,
Capital Normal University, Beijing, 100048, China}

\thanks{The first author acknowledges the support of NSF Grant DMS-1404596.
The second author acknowledges the support of NSFC Grant 11571019.}
\maketitle

\section{Introduction}\label{sec-Intro}

Let $(M,g)$ be a compact Riemannian manifold of dimension $n$ without boundary, for $n\ge 3$,
and $\Gamma$ be a smooth submanifold in $M$.
For $(M,g)=(S^n, g_{S^n})$,
Loewner and Nirenberg \cite{Loewner&Nirenberg1974} proved that there exists a
complete conformal metric on $S^n\setminus \Gamma$  with a {\it negative}
constant scalar curvature
if and only if dim$(\Gamma)>(n-2)/2$.
Aviles and McOwen \cite{AM1988DUKE} proved the similar result for the general manifold $(M,g)$.
As a consequence, we can take the dimension of the submanifold to be $n-1$
and conclude the following result: In any compact Riemannian manifold with boundary,
there exists a complete conformal metric with a negative constant scalar curvature. See \cite{AM1988DUKE}.
For convenience, we always take the constant scalar curvature to be $-n(n-1)$.
In this paper, we will study 
whether Ricci curvatures of such a metric remain negative.

For the case of positive scalar curvatures, the
existence and asymptotic behaviors of solutions have been extensively studied over the years.
We shall not discuss this case here, but refer to
\cite{Caffarelli1989}, \cite{Korevaar1999}, \cite{Mazzeo Pacard1999},
\cite{Mazzeo PollackUhlenbeck1996}, \cite{Schoen1988}, \cite{SchoenYau1988}.

There
are several classical results for metrics with negative Ricci curvatures.
Gao and Yau \cite{GaoYau} proved that
there exists a metric of negative Ricci curvature on every compact 3-dimensional
manifold without boundary.
Lohkamp \cite{Lohkamp5} generalized this to arbitrary dimensions and proved
that any manifold of dimension $n\geq 3$ (compact or not)
admits a complete metric of negative Ricci curvature.
Restricted to conformal metrics, by solving $\det(Ric)=constant$ with
a precise boundary asymptotics, Guan \cite{Guan2008} and Gursky, Streets and Warren \cite{M.Gursky1}
proved that there exists a complete
conformal metric
with negative Ricci curvature on a compact Riemannian manifold with boundary.

In the unit ball in the Euclidean space, the complete conformal metric with scalar
curvature $-n(n-1)$ is exactly the Poincar\'e metric of the unit ball model of the hyperbolic space
and has sectional curvatures $-1$ and Ricci curvatures $-(n-1)$. In particular, it has
negative sectional curvatures and Ricci curvatures.
A natural question is whether this remains true for the more general case; namely, whether
the complete conformal metric with a negative constant scalar curvature
in a compact Riemannian manifold with boundary has  negative sectional curvatures
or negative Ricci curvatures.
We point out that a straightforward calculation based on the polyhomogeneous expansion
established in \cite{ACF1982CMP} and \cite{Mazzeo1991} yields
that such a metric has sectional curvatures asymptotically equal to $-1$  near  boundary.
Our main concern is whether the negativity of the sectional curvatures or Ricci curvatures near boundary can
be carried over to the entire domain.

In view of the Poincar\'e metric in the unit ball model of the hyperbolic space, it is
reasonable to expect that the complete conformal metric with a negative constant scalar curvature
should have negative sectional curvatures in a domain close to the unit ball in the Euclidean space.
We will confirm this in this paper. In fact, we will prove an affirmative result for convex domains in
the Euclidean space.

\begin{theorem}\label{main reslut}
Let $\Omega\subset \mathbb{R}^{n}$ be a bounded convex domain, for $n\ge 3$,
and $g_\Omega$ be the complete conformal metric in $\Omega$
with the constant scalar curvature $-n(n-1)$.
Then, $g_\Omega$ has negative sectional curvatures in $\Omega$.
Moreover, $g_\Omega$ has Ricci curvatures  strictly less than $-{n}/{2}$ in $\Omega$.
\end{theorem}

The convexity assumption of the domain $\Omega$ is crucial. It allows us to
apply a convexity theorem by Kennington  \cite{Kennington1985} directly
to conformal factors. Theorem \ref{main reslut}
does not hold for general bounded domains in $\mathbb R^n$.
In fact, in certain bounded star-shaped domains,
the conformal metrics
may have arbitrarily large positive Ricci curvature components.
See Example \ref{exa-curves}.
Note that bounded convex domains and  bounded star-shaped domains have the
same topology.

Closely related to the negativity of the Ricci curvatures is whether there is a constant rank theorem for metrics
with negative Ricci curvatures, since it is already known that the Ricci curvatures are negative
near boundary. Caffarelli, Guan, and Ma \cite{Caffarelli2007}
proved a constant rank theorem for the $\sigma_k$-curvature equations under certain
positivity conditions on curvatures. However, their result is not applicable in our case.
Our strategy is to connect directly boundary curvatures of domains with the interior curvature tensors
of the complete conformal metrics.

\smallskip

We now turn our attention to bounded smooth domains which are sufficiently  ``far" from the unit ball.
According to Aviles and McOwen \cite{AM1988DUKE}, in order to have
a complete conformal metric with constant negative scalar curvature in $M\setminus \Gamma$,
it is required that dim$(\Gamma)>(n-2)/2$.
Closely related is a result proved by Mazzeo and Pacard \cite{Mazzeo Pacard1999}
that there exist complete conformal metrics in $S^n\setminus \Gamma$ with
constant {\it positive} scalar curvatures if dim$(\Gamma)\le (n-2)/2$. In view of these results,
we can ask what happens to Ricci curvatures of the complete conformal metrics with
scalar curvatures fixed at $-n(n-1)$ in domains
$\Omega\subset M$ whose $(n-1)$-dimensional
boundary is close to a smooth submanifold $\Gamma$ of dimension $\le (n-2)/2$.
Do Ricci curvatures have mixed signs as $\partial\Omega$ becomes close to a low
dimensional set, say a single point?


In this paper, we will construct domains where complete conformal metrics have large positive Ricci curvatures
in domains in compact Riemannian manifolds.

\begin{theorem}\label{thrm-large-positive-Ricci}
Let $(M, g)$ be a compact Riemannian manifold  of dimension $n\geq 3$ without boundary and
$\Gamma$ be a disjoint union of finitely many closed smooth embedded submanifolds
in $M$ of varying dimensions,
between $0$ and $(n-2)/2$.
Consider the following cases:

$\text{Case 1.}$ $\Gamma$ contains a submanifold of dimension $j$, with $1\le j\leq (n-2)/2$.

$\text{Case 2.}$ If $(M, g)$ is not conformally equivalent to the standard sphere $S^n$,
$\Gamma$ consists of finitely many points.

$\text{Case 3.}$  If $(M, g)$ is conformally equivalent to
$S^{n}$, $\Gamma$ consists of at least
two but only finitely many points.

Suppose that $\Omega_i$ is a sequence of increasing domains with smooth boundary in $M$
which converges to $M \setminus\
\Gamma$ and that $g_i$ is the complete conformal metric in $\Omega_i$
with the constant scalar curvature $-n(n-1)$.
Then, for sufficiently large $i$, $g_i$ has a positive Ricci curvature component somewhere in
$\Omega_{i}$. Moreover,
the maximal Ricci curvature in $\Omega_{i}$ diverges to $\infty$ as $i\to\infty$.
\end{theorem}

By the convergence of $\Omega_i$  to $M \setminus\Gamma$, we mean $\cup_{i=1}^\infty \Omega_i
=M\setminus \Gamma$ and, for any $\varepsilon>0$,
$\partial\Omega_i$ is in the $\varepsilon$-neighborhood of $\Gamma$ for all large $i$.
By convention, a zero dimensional submanifold is simply a point.

The difference between Case 2 and Case 3 in Theorem \ref{thrm-large-positive-Ricci}
lies on the number of isolated points when closed smooth embedded submanifolds
of positive dimension are absent from $\Gamma$.
On manifolds conformally
equivalent to the standard sphere, the number of the isolated points has to be at least two;
while on manifolds not conformally
equivalent to the standard sphere, we can allow one point.
Theorem \ref{thrm-large-positive-Ricci}
does not necessarily hold if $\Gamma$ consists of one point on manifolds conformally
equivalent to the standard sphere.
See Remark \ref{rmk-single-point-sphere}.
Such a difference demonstrates that the background manifolds also play a
decisive role in the issue studied in this paper.

As a consequence of Case 2, with $\Gamma$ consisting of just one point,
we have the following rigidity result.

\begin{theorem}\label{them-rigidity}
Let $(M, g)$ be a compact Riemannian manifold of dimension $n\geq3$ without boundary
and $x_0$ be a point in $M$.
Suppose that there exists a sequence $\Omega_i$ of increasing domains
with smooth boundary in $M$ which converges to $M \setminus\
\{x_0\}$, such that the complete conformal metric in $\Omega_i$
with the constant scalar curvature $-n(n-1)$ has uniformly bounded Ricci curvatures in
$\Omega_{i}$. Then, $M$ is conformally equivalent to the standard sphere $S^n$.
\end{theorem}


The set $\Gamma$ in Theorem \ref{thrm-large-positive-Ricci} resembles that in \cite{Mazzeo Pacard1999}.
The metric $g_i$ in $\Omega_i$ as in
Theorem \ref{thrm-large-positive-Ricci} is assumed to have a {\it negative} constant scalar curvature, $-n(n-1)$.
As $\Omega_i$ becomes close to $M\setminus \Gamma$, Ricci curvatures
split in sign. Some components become negatively large, while some others positively large.

The proof of Theorem \ref{thrm-large-positive-Ricci}
relies on a careful analysis of the Ricci curvatures of the complete conformal metrics near boundary.
The polyhomogeneous expansion provides correct values near boundary for applications of
the maximum principle. The Yamabe invariant of $(M,g)$ plays a crucial role and
determines behaviors of the convergence of the conformal factors.
Among the three cases listed in Theorem \ref{thrm-large-positive-Ricci}, Case 2 is the most difficult to
prove, especially when the Yamabe invariant is between zero and that of the standard sphere.
When $\Gamma$ consists of one point $x_0$, we need expansions of Green's functions.
If $n=3,4,5$, or $M$ is conformally flat in a neighborhood of $x_0$, we need to employ the
positive mass theorem. If $n\ge 6$ and  $M$ is not conformally flat in a neighborhood of $x_0$, we need to
distinguish the two cases $W(x_0)\neq 0$ and $W(x_0)=0$. Discussions for the latter case is much more
complicated than the former case.
The proof here seems to resemble the solution of the Yamabe problem, but with one twist.
In solving the Yamabe problem,
we can choose a point where the Weyl tensor is not zero in the case that $M$ is not conformally flat.
In our case,
$x_0$ is a given point and the Weyl tensor can be zero even if $M$ is not conformally
flat in a neighborhood of $x_0$.
Different vanishing orders of $W$ at $x_0$ requires different methods.
In fact, we also need to employ the positive mass theorem if
the Weyl tensor vanishes at $x_0$ up to a sufficiently high order.
The positive mass theorem has been known to be true if $3\leq n\leq 7$, or $M$ is locally conformally flat, or
$M$ is spin. (See \cite{LeeParker1987}, \cite{SchoenYau1979}, \cite{SchoenYau1994} and \cite{Witten1981}.)
These conditions might be technical and could be removed according to the recent papers
\cite{Lohkamp4} and
\cite{SchoenYau2017}. 
Refer to Remark \ref{unknown} on how the positive mass theorem
is used in the proof of
Theorem \ref{thrm-large-positive-Ricci}.



\smallskip

The paper is organized as follows.
In Section \ref{An Equivalent Form},
we discuss  some preliminary identities.
In Section \ref{sec-Conv}, we
study the Ricci curvatures of complete conformal metrics
in bounded convex domains in the Euclidean space and prove Theorem \ref{main reslut}.
In Section \ref{sec-compact-manifolds}, we study the Ricci curvatures of complete conformal
metrics in domains in compact manifolds and prove
Theorem \ref{thrm-large-positive-Ricci}.
In Section \ref{sec-Examples}, we present several examples in the Euclidean space.

We would like to thank Matthew Gursky for suggesting the problem studied in this paper and many helpful
discussions. Gursky graciously shared many of his stimulating computations with us.
We would also like to thank Yuguang Shi for helpful discussions.

\section{Preliminaries}\label{An Equivalent Form}

Let $(M,g)$ be a smooth Riemannian manifold of dimension $n$, for some $n\ge 3$, either compact
without boundary or noncompact and complete. Assume $\Omega\subset M$ is a smooth domain,
with an $(n-1)$-dimensional boundary. If
$(M, g)$ is noncompact, we assume, in addition, that $\Omega$ is bounded. We consider the following problem:
\begin{align}
\label{eq-MEq} \Delta_{g} u -\frac{n-2}{4(n-1)}S_gu&= \frac14n(n-2) u^{\frac{n+2}{n-2}} \quad\text{in }\,\Omega,\\
\label{eq-MBoundary}u&=\infty\quad\text{on }\partial \Omega,
\end{align}
where $S_g$ is the scalar curvature of $M$.
According to Loewner and Nirenberg \cite{Loewner&Nirenberg1974} for $(M,g)=(S^n, g_{S^n})$ and
Aviles and McOwen \cite{AM1988DUKE} for the general case,
\eqref{eq-MEq} and \eqref{eq-MBoundary} admits a unique positive solution.
We note that $u^{\frac{4}{n-2}}g$ is
the complete metric with a constant scalar curvature $-n(n-1)$ on $\Omega$.
Andersson, Chru\'sciel and Friedrich \cite{ACF1982CMP} and Mazzeo \cite{Mazzeo1991} established
the polyhomogeneous expansions for the solutions. For the first several terms, we have
$$u=d^{-\frac{n-2}{2}}\Big[1+\frac{n-2}{4(n-1)}H_{\partial\Omega}d+O(d^2)\Big],$$
where $d$ is the distance to $\partial\Omega$ and $H_{\partial\Omega}$ is the mean curvature of $\partial\Omega$
with respect to the interior
unit normal vector of $\partial\Omega$.
Set
\begin{equation}\label{eq-def-v}v=u^{-\frac{2}{n-2}}.\end{equation}
Then,
\begin{align}\label{eq-MEq-v}
v\Delta_{g} v +\frac{1}{2(n-1)}S_{ g} v^2&= \frac{n}{2}(|\nabla_{g} v|^2-1)\quad\text{in }\Omega, \\
\label{eq-MBoundary-v}v&=0\quad\text{on }\partial \Omega.
\end{align}
Moreover,
\begin{equation}\label{boundary expansion v}v=d-\frac{1}{2(n-1)}H_{\partial\Omega}d^2+O(d^3).\end{equation}
This implies
\begin{equation}\label{boundary expansion v-gradient}|\nabla_gv|=1\quad\text{on }\partial\Omega.\end{equation}
We will use this repeatedly later on.

Consider the conformal metric
\begin{equation}\label{eq-conformal-metric}
g_\Omega=u^{\frac{4}{n-2}}g=v^{-2}g.\end{equation}
For a unit vector $X$ of  $g$, $vX$ is a unit vector of $g_\Omega$.
Let $R_{ij}$ be the Ricci components of $g$ in a local frame for the metric $g$ and
$R^\Omega_{ij}$ be the Ricci components of $g_\Omega$
in the corresponding frame for the metric $g_\Omega$.
Then,
\begin{equation*}
{R}^\Omega_{kl}=v^2R_{kl}+(n-2)\big[vv_{,kl}-\frac12g_{kl}|\nabla _gv|^2\big]
+g_{kl}\big[v\Delta_gv-\frac{n}{2}|\nabla_gv|^2\big].
\end{equation*}
By \eqref{eq-MEq-v}, we have
\begin{equation}\label{Ricci cur  in v Manifold-1}
{R}^\Omega_{kl}=v^2R_{kl}+(n-2)\big[vv_{,kl}-\frac12g_{kl}|\nabla _gv|^2\big]
-g_{kl}\big[\frac{1}{2(n-1)}v^2S_g+\frac{n}{2}\big],
\end{equation}
or
\begin{equation}\label{Ricci cur  in v Manifold}
{R}^\Omega_{kl}=v^2R_{kl}-\frac{1}{2(n-1)}v^2g_{kl}S_g+(n-2)vv_{,kl}-\frac{n-2}{2}g_{kl}|\nabla_gv|^2-\frac{n}{2}g_{kl}.
\end{equation}
We emphasize that \eqref{Ricci cur  in v Manifold-1} and \eqref{Ricci cur  in v Manifold}
play important roles in the rest of the paper.
By \eqref{eq-MBoundary-v} and \eqref{boundary expansion v-gradient}, we obtain
$${R}^\Omega_{kl}=-(n-1)g_{kl}+O(d).$$
In other words, the Ricci curvatures of conformal metrics $g_\Omega$
are asymptotically equal to $-(n-1)$ near boundary.
We note that this holds in arbitrary smooth domains.

If $(M,g)=(\mathbb R^n, g_E)$,
then \eqref{eq-MEq} and \eqref{eq-MBoundary} reduce to
\begin{align}
\label{eq-MainEq} \Delta u  &= \frac14n(n-2) u^{\frac{n+2}{n-2}} \quad\text{in }\,\Omega,\\
\label{eq-MainBoundary}u&=\infty\quad\text{on }\partial \Omega.
\end{align}
In this case, the function $v$ given by \eqref{eq-def-v} satisfies
\begin{equation}\label{eq-transf} v\Delta v=\frac{n}{2}(|\nabla v |^2-1). \end{equation}
Let $g_\Omega$ be the metric given by \eqref{eq-conformal-metric} with $g=g_E$, i.e.,
$g_\Omega=v^{-2}g_E$.
Denote by $R^\Omega_{ijij}$ and $R^\Omega_{ij}$ the sectional curvatures and Ricci curvatures of
$g_\Omega$ in the orthonormal coordinates of $g_\Omega$, respectively.
Then, for $i\neq j$,
\begin{equation}\label{Sectional cur in v}
R^\Omega_{ijij}=vv_{ii}+vv_{jj}-|\nabla v|^{2}, \end{equation}
and, for any $i, j$,
\begin{equation}\label{Ricci cur  in v}
R^\Omega_{ij}=(n-2)vv_{ij}-\Big[\frac{n-2}{2}|\nabla v|^{2}+\frac{n}{2}\Big]\delta_{ij}.\end{equation}
Hence, for any $i\neq j$,
\begin{equation*}
R^\Omega_{ijij}=-1+O(d), \end{equation*}
and, for any $i, j$,
\begin{equation*}
R^\Omega_{ij}=-(n-1)\delta_{ij}+O(d).\end{equation*}
Note
$$v_i=-\frac{2}{n-2}u^{-\frac{2}{n-2}-1}u_i,$$
and
\begin{equation}\label{eq-relation-2derivatives}
v_{ij}=-\frac{2}{n-2}u^{-\frac{2}{n-2}}\Big(\frac{u_{ij}}{u}-\frac{n}{n-2}\frac{u_iu_j}{u^2}\Big).\end{equation}
We can also express $R^\Omega_{ijij}$ and $R^\Omega_{ij}$ in terms of $u$.

\section{Convex Domains in Euclidean Spaces}\label{sec-Conv}

In this section, we study Ricci curvatures and sectional curvatures of
the complete conformal metrics associated with the Loewner-Nirenberg problem
in bounded domains in the Euclidean space.
We will prove that the complete conformal metrics in bounded convex domains
have negative sectional curvatures.
The convexity assumption allows us to
apply a convexity theorem by Kennington  \cite{Kennington1985} directly
to conformal factors.

\begin{proof}[Proof of Theorem \ref{main reslut}]
Let $u$ be the solution of \eqref{eq-MainEq} and \eqref{eq-MainBoundary} in $\Omega$
and $v$ be given by \eqref{eq-def-v}.
Then, $g=v^{-2}g_E$ is the complete conformal metric in $\Omega$ with a
constant scalar curvature $-n(n-1)$. Denote by $R_{ijij}$ and $R_{ij}$
the sectional curvatures and Ricci curvatures of $g$ in the orthonormal coordinates
of $g$, given by
\eqref{Sectional cur in v} and \eqref{Ricci cur  in v}, respectively. Here, we suppress
$\Omega$ from the notations $g$, $R_{ij}$ and $R_{ijij}$.

By applying  the Laplacian operator to \eqref{eq-transf}, we get
$$v\Delta (\Delta v)+(2-n)\nabla v \nabla(\Delta v)= n|\nabla^2v|^2-(\Delta v)^2 \geq 0.$$
First, we assume that the boundary of $\Omega$ is smooth. By \eqref{boundary expansion v}, we have
$$\Delta v=-H_{\partial\Omega}-\frac{1}{n-1}H_{\partial\Omega}+O(d)=-\frac{n}{n-1}H_{\partial\Omega}+O(d).$$
Since $\Omega$ is convex, we have $\Delta v\le 0$ on $\partial \Omega$.
By the strong maximum principle, we obtain $\Delta v<0$ in $\Omega$.
Therefore, $|\nabla v| <1$ in $\Omega$ by \eqref{eq-transf}.
Next, we apply
Theorems 3.1 and 3.2 \cite{Kennington1985} in $\Omega$ and conclude that
$v$ is concave.
In fact, we write (2.13) as
$$\Delta v=\frac{n(|\nabla v |^2-1)}{2v}.$$
Then, we can verify directly the hypothesis (i) of Theorem 3.1 by $|\nabla v| <1$
and the hypothesis (ii) of Theorem 3.2 of [10], since
$|\nabla v| <1$, $v=0$ on $\partial \Omega$, and $\nabla v$ is the inner unit normal vector on $\partial \Omega$.

For general bounded convex domains, we can obtain the concavity of $v$ by approximations.

Since $v_{ii}\leq 0$, by \eqref{Sectional cur in v} and \eqref{Ricci cur  in v}, we get,
for  any $i\neq j$,
$$R_{ijij}\leq 0,$$
and, for any $i$,
$$R_{ii} \leq -\frac{n}{2}.$$
By \eqref{eq-relation-2derivatives}, we also have, for any $i$,
\begin{equation}\label{v-concave}\frac{u_{ii}}{u}-\frac{n}{n-2}\frac{u_i^2}{u^2} \geq 0. \end{equation}

Next, we prove that $R_{ijij}$ does not vanish in $\Omega$ for any $i\neq j$.
If $R_{ijij}= 0$ at some point $x_0\in\Omega$ for some $i\neq j$, then
$$\Big(\frac{u_{ii}}{u}-\frac{n}{n-2}\frac{u_i^2}{u^2}\Big)(x_0)=0,$$
 and
$$\nabla u(x_0)=0.$$
In fact, by (2.14) and $vv_{ii}\le0$, if $R_{ijij}(x_0)=0$, then
$\nabla v(x_0)=0$ and thus $\nabla u(x_0)=0$.
Hence, $u_{ii}(x_0)=0$.
Applying $\partial_i$ twice to the equation \eqref{eq-MainEq}, we get
$$\Delta u_{ii}=\frac14n(n+2)u^{\frac{n+2}{n-2}}\frac{u_{ii}}{u}+\frac{n(n+2)}{n-2}u^{\frac{n+2}{n-2}}\frac{u_i^2}{u^2}.$$
Combining with \eqref{v-concave}, we have
$$\Delta u_{ii}-\frac14(n+2)(n+4)u^{\frac{4}{n-2}}u_{ii}
=(n+2)u^{\frac{n+2}{n-2}}\Big(\frac{n}{n-2}\frac{u_i^2}{u^2}-\frac{u_{ii}}{u}\Big)\leq0.$$
By the strong maximum principle, we have $u_{ii}\equiv0 $ in $\Omega$.
On the other hand, by $u_{i}(x_0)=0$, we get $u_{i}\equiv0$ on $\Omega \cap \{x_0+te_i|t \in \mathbb{R}\}$.
Therefore, $u$ is constant on $\Omega \cap \{x_0+te_i|t \in \mathbb{R}\}$.
This leads to a contradiction.
Therefore, we have, for any $i\neq j$,
$$R_{ijij}< 0.$$
Similarly, we have, for any $i$,
$$R_{ii} <-\frac{n}{2}.$$
This completes the proof.
\end{proof}

We point out that the upper bounds of sectional curvatures and Ricci curvatures
in Theorem \ref{main reslut} are given by strict inequalities and, in fact, are optimal.
To see this, set
$$D=\big\{(x_1, \cdots, x_n)|-1<x_n<1\big\}\subset\mathbb R^n.$$
Let $g_D$ be the complete conformal metric
with the constant scalar curvature $-n(n-1)$ in $D$, and $R^{D}_{ijij}$ and
$R^{D}_{ii}$ be the sectional curvatures and Ricci curvatures of $g_D$, respectively.
Then, $R^{D}_{ijij}(0)=0$, for $i\neq j$, $i,j \neq n$,
and $R^{D}_{ii}(0)=-{n}/{2}$, for $i \neq n$.
Set $\Omega_{R}=D\bigcap B_{R}$.
Then, $R^{\Omega_R}_{ijij}(0)\rightarrow 0$, for $i\neq j$, $i,j \neq n$,
and $R^{\Omega_R}_{ii}(0)\to-{n}/{2}$ for $i \neq n$, as $R\to\infty$.

\section{Domains in Compact Manifolds}\label{sec-compact-manifolds}

In this section, we discuss domains in compact Riemannian manifolds without boundary.
We construct domains with boundary close to certain sets of low dimension such that the complete
conformal metrics with a negative constant scalar curvature have positive Ricci components somewhere.
Throughout this section, the Yamabe invariant plays a crucial role. It determines convergence behaviors
of conformal factors and, as a consequence, the methods to be employed.
In certain cases, we need to employ expansions of the
Green's function, and also the positive mass theorem.

Suppose $(M, g)$ is a compact Riemannian manifold of dimension $n\geq 3$ without boundary.
The Yamabe invariant of $M$ is given by
$$\lambda(M, [g])=\inf\Big\{
\frac{\int_M  (  |\nabla_g \phi|^2 +\frac{n-2}{4(n-1)} S_g\phi^2 )dV_g   }{(\int_M \phi^{\frac{2n}{n-2}}dV_g)^{\frac{n-2}{n}}}
\Big|\,\phi \in C^{\infty}(M),\phi>0 \Big\}.$$
The conformal Laplacian of $(M,g)$ is given by
$$L_g=-\Delta_g +\frac{n-2}{4(n-1)}S_g.$$
For any function $\psi$ in $M$, we have
$$L_g(u\psi)= u^{\frac{n+2}{n-2}}L_{u^{\frac{4}{n-2}}g}(\psi ).$$

We first prove a convergence result which plays an important role in this section.
According to signs of Yamabe invariants, conformal factors exhibit different convergence behaviors.
We note that the maximum principle is applicable to the operator $L_g$ if $S_g\ge 0$.

\begin{lemma}\label{lemme-convergence}
Suppose $(M, g)$ is a compact Riemannian manifold of dimension $n\geq 3$ without boundary,
with a constant scalar curvature $S_g$, and
$\Gamma$ is a closed smooth submanifold of dimension $d$ in $M$, $0\le d\leq \frac{n-2}{2}$.
Suppose $\Omega_i$ is a sequence of increasing domains with smooth boundary in $M$
which converges to $M \setminus\Gamma$.
Let  $u_i$ be the solution
of \eqref{eq-MEq} and \eqref{eq-MBoundary} in $\Omega_i$.
Then, for any positive integer $m$, if $S_g\ge 0$, 
\begin{equation}\label{u_i go to zero-plus}
u_i\rightarrow 0\quad\text{in }C^{m}_{\mathrm{loc}}( M \setminus\Gamma)
\text{ as }i\rightarrow\infty,\end{equation}
and, if $S_g<0$, 
\begin{equation}\label{u_i go to zero-minus}
u_i\rightarrow \Big(\frac{-S_g}{n(n-1)}\Big)^{\frac{n-2}{4}}
\quad\text{in }C^{m}_{\mathrm{loc}}( M \setminus\Gamma)
\text{ as }i\rightarrow\infty.\end{equation}
\end{lemma}

\begin{proof} We first consider the case $S_g\ge 0$. 
By the maximum principle, we have $u_i \geq u_{i+1}$ in $\Omega_i$. It is straightforward to verify, for any $m$,
$$u_i\rightarrow u\quad\text{in }C^{m}_{\mathrm{loc}}( M \setminus\Gamma)
\text{ as }i\rightarrow\infty,$$
where $u$ is a nonnegative solution of \eqref{eq-MEq} in $M \setminus\Gamma$.
By the second part of \cite{AM1988DUKE} (Page 398), $u$ is bounded.
Let $\rho(x)$ be a positive smooth function in $ M \setminus\Gamma$
which equals to $\text{dist}(x, \Gamma)$ in a neighborhood of $\Gamma$ in $M$.
Then,
$$\rho(x)\Delta_g \rho(x)\rightarrow n-d-1\quad\text{as }x\rightarrow\Gamma.$$
Take $\epsilon_0>0$ sufficiently small. Then,
$$\Delta_g \rho^{-\frac{n-2}{2}+\epsilon_0}=\Big(-\frac{n-2}{2}+\epsilon_0\Big)\rho^{-\frac{n+2}{2}+\epsilon_0}
\Big(\rho\Delta_g \rho -\big(\frac{n}{2}-\epsilon_0\big)|\nabla_g\rho|^2\Big).$$
Since $d\leq\frac{n-2}{2}$, we have $\Delta  \rho^{-\frac{n-2}{2}+\epsilon_0}<0$ near $\Gamma$. For any
$\epsilon>0$, we can find $\delta <\epsilon$ sufficiently small such that
$$\Delta_g(\delta \rho^{-\frac{n-2}{2}+\epsilon_0}+\epsilon)
-S_g (\delta \rho^{-\frac{n-2}{2}+\epsilon_0}+\epsilon)
\leq \frac{n-2}{4}(\delta \rho^{-\frac{n-2}{2}+\epsilon_0}+\epsilon)^{\frac{n+2}{n-2}}\quad\text{in }M\setminus\Gamma.$$
By the maximum principle, we have
$$u \leq \delta \rho^{-\frac{n-2}{2}+\epsilon_0}+\epsilon\quad\text{in }M\setminus\Gamma.$$
This implies $u\equiv 0$. In conclusion, we obtain \eqref{u_i go to zero-plus}.

We now consider the case $S_g<0$. 
We first prove $\Delta_g u_i\geq 0$ in $\Omega_i,$
or equivalently
\begin{equation}\label{eq-AlgebraicRelation}u_i\geq
\Big(\frac{-S_g}{n(n-1)}\Big)^{\frac{n-2}{4}}
\quad\text{in }\Omega_i.\end{equation}
If \eqref{eq-AlgebraicRelation} is violated somewhere,
then $u_i$ must assume its minimum at some point $x_0$ in the set
$$ \Big\{x\in \Omega_i :\, \frac14n(n-2) u_i^{\frac{n+2}{n-2}} +\frac{n-2}{4(n-1)}S_gu_i< 0 \Big\}.$$
On the other hand, we have $
\Delta_g u_i(x_0)\geq 0$, which leads to a contradiction.
By taking a difference, we have
$$\Delta_g (u_{i+1}-u_{i})=c_i(u_{i+1}-u_{i})\quad \text{in } \Omega_{i},$$
where $c_i$ is a nonnegative function in $\Omega_{i}$ by \eqref{eq-AlgebraicRelation}.
The maximum principle implies $u_{i+1}\leq u_{i}$ in $\Omega_i$.
Then, for any $m$,
$$u_i\rightarrow u\quad\text{in }C^{m}_{\mathrm{loc}}( M \setminus\Gamma)
\text{ as }i\rightarrow\infty,$$
where $u$ is a solution of \eqref{eq-MEq} in $M \setminus\Gamma$.
By \eqref{eq-AlgebraicRelation}, we have
\begin{equation}\label{eq-u lower bounded} u\geq
\Big(\frac{-S_g}{n(n-1)}\Big)^{\frac{n-2}{4}}\quad\text{in }M\setminus\Gamma.\end{equation}
For $\epsilon>0$ sufficiently small, let $u^{\epsilon}_i$ be the solution of
\begin{align}
\label{eq-MEq2} \Delta_g u^{\epsilon}_i &= \epsilon\frac{n(n-2)}{4}( u^{\epsilon}_i)^{\frac{n+2}{n-2}}
\quad\text{in }\Omega_i,\\
\label{eq-MBoundary2}u^{\epsilon}_i &=\infty\quad\text{on }\partial \Omega_i.
\end{align}
The existence of $u^{\epsilon}_i$ can be obtained by the standard method.
More specifically, for each integer $j$, we solve
\begin{align}
\label{eq-MEq21} \Delta_g u^{\epsilon,j}_i &= \epsilon\frac{n(n-2)}{4}( u^{\epsilon,j}_i)^{\frac{n+2}{n-2}}
\quad\text{in }\Omega_i,\\
\label{eq-MBoundary21}u^{\epsilon,j}_i &=j\quad\text{on }\partial \Omega_i.
\end{align}
By the maximum principle, we have $u^{\epsilon,j}_i \leq u^{\epsilon,k}_i$ if $j\le k$.
For any $x_0 \in \Omega_i$, choose normal coordinates near $x_0$. Then, it is easy to check that
$$u_{r, x_0}(x)=2\epsilon^{-\frac{n-2}{4}}\left(\frac{2r}{r^{2}-|x|^{2}}\right)^{\frac{n-2}{2}}$$ is a supersolution of \eqref{eq-MEq2}
when $r$ is sufficiently small, depending on $x_0$. Hence, by the maximum principle, we have for each point $x$,
$u^{\epsilon,j}_i(x)\le C(x)$, independent of $j$. Therefore,
by standard estimates, $u^{\epsilon,j}_i$ converges to some  $u^{\epsilon}_i$ in
$C^{m}_{\mathrm{loc}}(\Omega_i)$ as $j\rightarrow\infty$ for any $m$, and
$u^{\epsilon}_i \in C^{\infty}( \Omega_i)$ is a solution of \eqref{eq-MEq2}-\eqref{eq-MBoundary2}.

By the same method as in the proof of the case $S_g\geq0$, we obtain, for any $m$,
$$u^{\epsilon}_i\rightarrow 0\quad\text{in }C^{m}_{\mathrm{loc}}( M \setminus\Gamma)
\text{ as }i\rightarrow\infty.$$
Next, we can verify
\begin{align*}
\Delta_g \Big[u^{\epsilon}_i+  \Big(\frac{-S_g}{(1-\epsilon)n(n-1)}\Big)^{\frac{n-2}{4}} \Big]
&\le\frac{n-2}{4(n-1)}S_g
\Big[u^{\epsilon}_i+ \Big(\frac{-S_g}{(1-\epsilon)n(n-1)}\Big)^{\frac{n-2}{4}} \Big] \\
&\qquad +\frac{n(n-2)}{4} \Big[
u^{\epsilon}_i+  \Big(\frac{-S_g}{(1-\epsilon)n(n-1)}\Big)^{\frac{n-2}{4}} \Big]^{\frac{n+2}{n-2}}.
\end{align*}
To prove this, we simply split the last term according to $1=\epsilon+(1-\epsilon)$. Then,
\begin{align*}
&\frac{n-2}{4(n-1)}S_g
\Big[u^{\epsilon}_i+ \Big(\frac{-S_g}{(1-\epsilon)n(n-1)}\Big)^{\frac{n-2}{4}} \Big]\\
&\qquad\qquad+\frac{n(n-2)}{4} \Big[
u^{\epsilon}_i+  \Big(\frac{-S_g}{(1-\epsilon)n(n-1)}\Big)^{\frac{n-2}{4}} \Big]^{\frac{n+2}{n-2}}
-\epsilon\frac{n(n-2)}{4}( u^{\epsilon}_i)^{\frac{n+2}{n-2}}\\
&\qquad\geq\frac{n(n-2)}{4} \Big[
u^{\epsilon}_i+  \Big(\frac{-S_g}{(1-\epsilon)n(n-1)}\Big)^{\frac{n-2}{4}} \Big] \cdot\\
&\qquad\qquad \Big\{(1-\epsilon)
\Big[u^{\epsilon}_i+ \Big(\frac{-S_g}{(1-\epsilon)n(n-1)}\Big)^{\frac{n-2}{4}} \Big]^{\frac{4}{n-2}}
+\frac{1}{n(n-1)}S_g\Big\},
\end{align*}
which is nonnegative. By the maximum principle, we have
$$u_i \leq u^{\epsilon}_i+  \Big(\frac{-S_g}{(1-\epsilon)n(n-1)}\Big)^{\frac{n-2}{4}}
\quad\text{in }\Omega_i,$$
where we can verify the boundary condition by the polyhomogeneous expansions of $u_i$ and $u^{\epsilon}_i$.
Therefore, we have
\begin{equation}\label{eq-u upper bounded}
u\leq  \Big(\frac{-S_g}{(1-\epsilon)n(n-1)}\Big)^{\frac{n-2}{4}}\quad\text{in }M\setminus\Gamma. \end{equation}
This holds for any $\epsilon\in (0,1)$.
Combining \eqref{eq-u lower bounded} and \eqref{eq-u upper bounded}, we obtain
$$u=\Big(\frac{-S_g}{n(n-1)}\Big)^{\frac{n-2}{4}}.$$
In conclusion, we obtain \eqref{u_i go to zero-minus}. \end{proof}

A similar result holds if the scalar curvature has a fixed sign, not necessarily constant.

Now, we study the case that the boundary is close to a closed smooth submanifold of low dimension.
The result below holds for all compact manifolds without boundary, but
different signs of the Yamabe invariants require different methods, mostly due to the different convergence
behaviors as in Lemma \ref{lemme-convergence}.

\begin{theorem}\label{prop-general mainfold}
Suppose $(M, g)$ is a compact Riemannian manifold of dimension $n\geq 3$ without boundary and
$\Gamma$ is a closed smooth submanifold of dimension $d$ in $M$, $1\le d\leq \frac{n-2}{2}$.
Suppose $\Omega_i$ is a sequence of increasing domains with smooth boundary in $M$
which converges to $M \setminus\Gamma$ and $g_i$ is the complete conformal metric in $\Omega_i$
with the scalar curvature $-n(n-1)$.
Then, for sufficiently large $i$, $g_{i}$
has a positive Ricci curvature component somewhere in
$\Omega_{i}$. Moreover,
the maximal Ricci curvature of $g_i$ in $\Omega_{i}$ diverges to $\infty$ as $i\to\infty$.
\end{theorem}

\begin{proof} Let  $u_i$ be the solution
of \eqref{eq-MEq} and \eqref{eq-MBoundary} in $\Omega_i$ and set $v_i=u_i^{-\frac{2}{n-2}}$.
Then,
$$g_{i}=u_i^{\frac{4}{n-2}}g=v_i^{-2}g.$$
By the solution of the Yamabe problem, we can assume the scalar curvature $S_g$
of $M$ is the constant $\lambda(M,[g])$.
Since $M$ is compact, we can take $\Lambda>0$ such that
$$|R_{ij}|\leq \Lambda g_{ij}.$$
We now discuss two cases according to the sign of $S_g$.

{\it Case 1.} We first consider the case $S_g\geq0$.
By Lemma \ref{lemme-convergence}, for any $m$,
$$u_i\rightarrow 0\quad\text{in }C^{m}_{\mathrm{loc}}( M \setminus\Gamma)
\text{ as }i\rightarrow\infty,$$
and hence
$$v_i \text{ diverges to $\infty$ locally uniformly in $M \setminus\Gamma$
 as $i\rightarrow\infty$}.$$
We now consider two subcases.

\smallskip

{\it Case 1.1.} $\Gamma$ is not totally geodesic.
For any $\epsilon>0$, there exist two points $p,q \in \Gamma$,
such that the length of the shortest geodesic $\sigma_{pq}$ connecting $p$ and $q$
is less than $\epsilon$ and $\sigma_{pq}\bigcap \Gamma=\{p, q\}$.
When $\epsilon$ is sufficiently small,
we can assume $q$ is located in a small neighborhood of $p$ covered by normal coordinates.
Without loss of generality, we assume $p=0$ and $q=Le_n$.

For $i$ large,
set $p_i=\widehat{t}_ie_n$ and $q_i=\widetilde{t}_ie_n$, where
\begin{align*}\widehat{t}_i&=\min \{t'\in[0,{L}/{2}] |te_n \in\Omega_i , \text{ for any } t\in (t',{L}/{2}]\},\\
\widetilde{t}_i&= \max \{t'\in[{L}/{2},L]|te_n \in\Omega_i , \text{ for any } t\in [{L}/{2},t')\}.\end{align*}
Then, $p_i, q_i\in \partial\Omega_i$. By the convergence of $\Omega_i$ to $M\setminus \Gamma$,
we have
$$p_i\rightarrow p,\quad q_i\rightarrow q.$$
By the polyhomogeneous expansions of $v_i$, we have
$$|\partial_n v_i(p_i)|\leq C_i \quad\text{and}\quad |\partial_n v_i(q_i)|\leq C_i,$$
where $C_i$ is some positive constant which converges to 1 as $i\rightarrow\infty$ and $\epsilon\rightarrow 0$.

Since $v_i(Le_n/2)\rightarrow\infty$
as $i\rightarrow\infty$, for $i$ large,
we can take $t_i \in (\widehat{t}_i,\widetilde{t}_i)$ such that, for any $t\in (\widehat{t}_i,\widetilde{t}_i)$,
$$\partial_nv_i(te_n)\leq \partial_nv_i(t_ie_n).$$
Then,
$$\partial_nv_i(t_ie_n)>\frac{v_i(\frac{L}{2}e_n)-0}{\frac{L}{2}}\geq \frac{2}{L}v_i(\frac{L}{2}e_n),$$
and $$\partial_{nn}v_i( t_ie_n)=0.$$
We also have \begin{equation}\label{eq-v_icontrol}|v_i( t_ie_n)|\leq \frac{L}{2}\partial_nv_i(t_ie_n).\end{equation}
Denote by $R^i_{nn}$ the Ricci curvature of $g_i$
acting on the unit vector $v_i \frac{\partial}{\partial x^n}$ with respect to the metric $g_i$.
By \eqref{Ricci cur  in v Manifold}, we have, at $t_ie_n$,
$$R^i_{nn}\le v_i^2|R_{nn}|-\frac{n-2}{2}(\partial_nv_i)^2\le
\Big[\frac{L^2}{4}|R_{nn}|-\frac{n-2}{2}\Big](\partial_nv_i)^2\to-\infty,$$
if $L$ is sufficiently small.
Hence, some component of the Ricci curvature of $g_i$
at the point $t_ie_n$ diverges to $\infty$ as $i\rightarrow\infty$.


\smallskip

{\it Case 1.2.} $\Gamma$ is totally geodesic.
Fix a point $x_0 \in \Gamma$ and choose normal coordinates near $x_0$
such that $x_{0}=0$ and
$\Gamma$ near $x_0$ is given by $x_i=0$, $i=1,..,n-d$.
Consider the curve $\sigma$ given by
$$\sigma(t)=(\sqrt{R^2-t^2}-\sqrt{R^2-\epsilon^2}, 0, \cdots, 0, t)\quad\text{for }t \in [-\epsilon,\epsilon],$$
where $R$ is some sufficiently large constant and
$\epsilon$ is some sufficiently small constant such that
$\sigma\bigcap \Gamma=\{ \sigma(-\epsilon), \sigma(\epsilon)\}$.

For $i$ large, set $p_i=\sigma(\widehat{t}_i)$ and $q_i=\sigma( \widetilde{t}_i)$,  where
\begin{align*}\widehat{t}_i&=\min \{t'\in[-\epsilon,0] |\sigma(t) \in \Omega_i,
\text{ for any } t\in ( t',0]\},\\
\widetilde{t}_i&=\max \{t'\in [0,\epsilon] |\sigma(t) \in \Omega_i, \text{ for any } t\in [0,t')\}.\end{align*}
Then, $p_i, q_i\in\partial\Omega_i$ and
$$p_i\rightarrow \sigma(-\epsilon),
\quad q_i\rightarrow \sigma(\epsilon).$$
By the polyhomogeneous expansion of $v_i$,
we have
$$|\partial_{n} v_i(\sigma(\widehat{t}_i))|\leq C_i\quad\text{and}\quad
|\partial_{n} v_i(\sigma(\widetilde{t}_i))|\leq C,$$
where $C$ is some positive bounded  constant independent of $i$.

Consider the single variable function $(v_i\circ\sigma)(t)$.
Since $(v_i\circ\sigma)(0)\rightarrow\infty$ as $i\rightarrow\infty$, for $i$ large,
we can take $t_i \in (\widehat{t}_i,\widetilde{t}_i)$ such that, for any $t\in (\widehat{t}_i,\widetilde{t}_i)$,
$$\partial_{t}(v_i\circ\sigma)(t)\leq \partial_{t} (v_i\circ\sigma)(t_i).$$
Then,
$$\partial_{t} (v_i\circ\sigma)(t_i)>\frac{1}{\epsilon} (v_i\circ\sigma)(0),$$
and
\begin{equation}\label{eq-second deriv}\partial_{tt} (v_i\circ\sigma)(t_i)=0.\end{equation}
We also have
\begin{equation}\label{eq-v_icontrol2}|(v_i\circ\sigma)(t_i)|\leq
\epsilon\partial_{t} (v_i\circ\sigma)(t_i).\end{equation}
Note that
$$\partial_{t}(v_i\circ\sigma)(t_i)=
(\partial_n v_i)(\sigma(t_i))-\frac{t_i}{\sqrt{R^2-t_i^2}}(\partial_1 v_i)(\sigma(t_i)).$$
Set
$$\nu_i=\frac{\partial}{\partial x_n}-\frac{t_i}{\sqrt{R^2-t_i^2}}\frac{\partial}{\partial x_1}.$$
By \eqref{eq-second deriv}, we have
$$( \partial_{\nu_i\nu_i}v_i)(\sigma(t_i))
=\Big(-\frac{1}{\sqrt{R^2-t_i^2}}+\frac{t_i^2}{(R^2-t_i^2)^{\frac{3}{2}}}\Big)\partial_1 v_i(\sigma(t_i)).$$
Hence, $( \partial_{\nu_i\nu_i}v_i)(\sigma(t_i))$ is sufficiently small compared with $( |\nabla v_i|)(\sigma(t_i))$,
for $R$ sufficiently large and
$\epsilon$ sufficiently small.
Write $g_{\nu_i\nu_i}=g (\nu_i, \nu_i)$ and
denote by $R^i_{\nu_i\nu_i}$ the Ricci curvature of $g_i$
acting on the unit vector $\frac{v_i\nu_i}{\sqrt{g_{\nu_i\nu_i}}}$ with respect to the metric $g_i$.
Similarly as in Case 1.1, we can verify at the point $\sigma(t_i)$,
$R^i_{\nu_i\nu_i}$
diverges to $ -\infty$ as $i\rightarrow\infty$. Hence,  some component of the Ricci curvature of $g_i$
at the point $\sigma(t_i)$ diverges to $\infty$ as $i\rightarrow\infty$.

\smallskip
{\it Case 2.} We now consider the case $S_g<0$.
By Lemma \ref{lemme-convergence},  for any $m$,
$$u_i\rightarrow \Big(\frac{-S_g}{n(n-1)}\Big)^{\frac{n-2}{4}}
\quad\text{in }C^{m}_{\mathrm{loc}}( M \setminus\Gamma)
\text{ as }i\rightarrow\infty,$$
and hence
\begin{equation}\label{eq-v convergence}
v_i \rightarrow
\Big(\frac{-S_g }{n(n-1)}\Big)^{-\frac{1}{2}} \quad\text{in }C^{m}_{\mathrm{loc}}( M \setminus\Gamma)
\text{ as }i\rightarrow\infty.\end{equation}
Fix a point $x_0 \in \Gamma$ and choose normal coordinates in a small neighborhood of
$x_{0}$ such that $x_{0}=0$
and $x_n$-axis is a normal geodesic of $\Gamma$ near $x_0$.
Take $\epsilon>0$ sufficiently small.
For $i$ large,
set $p_i=t_ie_n$, where
$$t_i=\min \{t'\in[0,\epsilon] |te_n \in \Omega_i, \text{ for any }t\in (t', \epsilon]\}.$$
Then, $p_i\in\partial\Omega_i$ and
$$p_i\rightarrow 0.$$
By the polyhomogeneous expansion of $v_i$, we have
$$|\partial_n v_i(p_i)|\leq C_i,$$ where $C_i$ is some positive constant which converges to 1 as $i\rightarrow\infty$.
By \eqref{eq-v convergence},
$$\frac{\partial v_i}{\partial x_n}(\epsilon e_n )\rightarrow0\quad\text{as }i\rightarrow\infty.$$
For $i$ large, we take $\widetilde{t}_i \in (t_i,\epsilon)$ such that, for any $t\in (t_i,\epsilon)$,
$$\partial_nv_i(te_n)\leq \partial_nv_i(\widetilde{t}_ie_n).$$
Then,
$$\partial_nv_i(\widetilde{t}_ie_n)>\frac{v_i(\epsilon e_n)-0}{\epsilon-t_i}>
\frac{1}{2}\epsilon^{-1}\Big(\frac{-S_g }{n(n-1)}\Big)^{-\frac{1}{2}},$$
and $$\partial_{nn}v_i(\widetilde{ t}_ie_n)=0.$$
Denote by $R^i_{nn}$ the Ricci curvature of $g_i$
acting on the unit vector $v_i \frac{\partial}{\partial x^n}$ with respect to the metric $g_i$.
Similarly, by \eqref{Ricci cur  in v Manifold} at the point $\widetilde{t}_ie_n$,
$R_{nn}^i\le -C\epsilon^{-2}$, for all large $i$, for some positive constant $C$
independent of $i$ and $\epsilon$. By choosing appropriate $\epsilon$, we have the desired result.
\end{proof}

Next, we discuss the case that the boundary is close to a point $x_0$. The proof of the next result is
rather delicate if the Yamabe invariant is positive, in which case expansions of the
Green's function play an essential role.
We need to employ the positive mass theorem if the manifold has a dimension 3, 4, or 5, or
is locally conformally flat.
In the case that $n\ge 6$ and $M$ is not conformally flat in a neighborhood of $x_0$, we need
to analyze Weyl tensors and distinguish two cases $W(x_0) \neq 0$ and $W(x_0)=0$.
The proof for the case $W(x_0)=0$ is quite delicate. It is worth to emphasize that the Weyl tensor can
be zero at $x_0$ even if $M$ is not conformally flat in a neighborhood of $x_0$.
Different vanishing orders of $W$ at $x_0$ requires different methods.
In fact, we also need to employ the positive mass theorem if
the Weyl tensor vanishes at $x_0$ up to a sufficiently high order.

\begin{theorem}\label{blow up-one point}
Let $(M, g)$ be a compact Riemannian manifold of dimension $n\geq 3$ without boundary, with
$\lambda(M,[g])< \lambda(S^n,[g_{S^n}])$,
where $S^n$ is the sphere with its standard metric $g_{S^n}$,
and let $x_0$ be a point in $M$. Suppose that $\Omega_i$ is a sequence of increasing domains
with smooth boundary in $M$ which converges to $M \setminus\
\{x_0\}$ and $g_{i}$ is the complete conformal metric in $\Omega_i$ with the scalar curvature
$-n(n-1)$. Then, for $i$ sufficiently large, $g_i$
has a positive Ricci curvature component somewhere in
$\Omega_{i}$. Moreover,
the maximal Ricci curvature of $g_i$ in $\Omega_{i}$ diverges to $\infty$ as $i\to\infty$.
\end{theorem}

\begin{proof}
Let $u_i$ be the solution of \eqref{eq-MEq} and
\eqref{eq-MBoundary} in $\Omega_i$ and set $v_i=u_i^{-\frac{2}{n-2}}$.
Then,
$$g_{i}=u_i^{\frac{4}{n-2}}g=v_i^{-2}g.$$
We consider several cases according to the sign of the Yamabe invariant $\lambda(M,[g])$.

\smallskip
{\it Case 1.} We first consider $\lambda(M,[g])<0$.
We point out that the proof of Case 2 of Theorem \ref{prop-general mainfold} can be adapted
to yield the conclusion.

\smallskip
{\it Case 2.} Next, we consider $\lambda(M,[g])=0$. As in the proof of  Theorem \ref{prop-general mainfold},
we assume the scalar curvature of $M$ is 0 and
$$|R_{ij}|\leq \Lambda g_{ij}.$$
Let $\delta$ be some small positive constant such that
$\Lambda\delta<{1}/{10}$ and
there exist normal coordinates in $B_{\delta}(x_0 )$.

Take a sufficiently small $r>0$ with $r\leq \delta$.
Since $\Omega_i\rightarrow M \setminus\ \{x_0\}$, we have
$M \setminus\ B_{r}(x_0 )\subset\subset \Omega_i$ for $i$ large. For such $i$, by
the Harnack inequality, we have
$$\max u_i \leq C \min u_i
\quad\text{in }M \setminus\ B_{r}(x_0 ),$$
where $C$ is some positive constant depending only on $n$, $M$ and $r$.
Then for $i$ sufficiently large, by \eqref{u_i go to zero-plus}, we have
\begin{equation}\label{u_i grad estimate} |\nabla_g u_i| \leq C(u_i+ u_i^{\frac{n+2}{n-2}})\leq Cu_i
\quad\text{in }M \setminus\ B_{r}(x_0 ).\end{equation}
We denote by $m_i$ the minimum of $u_i$ in $\Omega_i$.
With the definition of $v_i$, we have, for $i$ sufficiently large,
\begin{equation}\label{v_i grad estimate} |\nabla_g v_i| \leq C v_i \leq C m_i^{-\frac{2}{n-2}}
\quad\text{in }M \setminus\ B_{r}(x_0 ),\end{equation}
where $C$ is some positive constant depending only on $n$, $M$ and $\delta$.
Set
$$A_i=\{x\in \Omega_i|\, u_i(x)<2m_i\}.$$
Then, for any fixed $ r>0$ with $r\leq \delta$ and any $i$ sufficiently large, we have $A_i\cap
B_{r}(x_0 ) \neq \emptyset.$ Otherwise, by the maximum principle, we have
$$u_i(x) \geq 2m_i-Cm_i^{\frac{n+2}{n-2}}\quad\text{in }A_i,$$
where $C$ is some constant depending only on $n$, $M$ and $r$.
Hence,  $$m_i \geq 2m_i-Cm_i^{\frac{n+2}{n-2}}.$$
Note that $m_i\rightarrow0$ as $i\rightarrow\infty$, which leads to a contradiction.

By $v_i=0$ on $\partial \Omega_i$,
we have, for any fixed $ r>0$ with $r\leq \delta$ and for any $i$ sufficiently large,
$$|\nabla_g v_i| \geq \frac{1}{r}(2m_i)^{-\frac{2}{n-2}}\quad\text{somewhere in }\Omega_i\cap  B_{r}(x_0 ).$$
Therefore, for $i$ sufficiently large, $|\nabla_g v_i|$ must assume its maximum at
$p_i \in \Omega_i\cap  B_{\delta}(x_0 ) $.
Write $\nu_i=\frac{\nabla_g v_i}{|\nabla_g v_i|}$ and
denote by $R^i_{\nu_i\nu_i}$ the Ricci curvature of $g_i$
acting on the unit vector $v_i\nu_i$ with respect to the metric $g_i$.
Then, we can proceed as in the proof of Theorem \ref{prop-general mainfold}
to verify  at the point $p_i$,
$R^i_{\nu_i\nu_i}$
diverges to $ -\infty$ as $i\rightarrow\infty$.
Hence, some component of the Ricci curvature of $g_i$
at the point $p_i$ diverges to $\infty$ as $i\rightarrow\infty$.

\smallskip
{\it Case 3.} We now consider the case $\lambda(M,[g])>0$.
In this case, there exists $G_{x_0}\in C^{\infty}(M \setminus\
\{x_0\})$, the  Green's function for the conformal Laplacian $L_g$, such that
$$L_gG_{x_0}=(n-2)\omega_{n-1}\delta_{x_0},\,\, G_{x_0}>0,$$
where $\omega_{n-1}$ is the volume of $S^{n-1}$.
Up to a conformal factor, we can assume $(M, g)$ has conformal normal coordinates
near $x_0$. See \cite{LeeParker1987} (Page 69)
or \cite{SchoenYau1994}, chapter 5.
We can perform a conformal blow up at $x_0$ to obtain an asymptotic flat and scalar flat manifold
by using $G_{x_0}$. Specifically, if we define the metric $\widetilde{g}
 = G_{x_0}^{\frac{4}{n-2}}g$ on $\widetilde{M} =M \setminus\
\{x_0\}$, then,  $(\widetilde{M}, \widetilde{g})$ is an asymptotically flat and scalar flat manifold,
and $\widetilde{g}$ has an asymptotic expansion near infinity.
See \cite{LeeParker1987} (Page 64-65), or
\cite{SchoenYau1994}, chapter 5.

Set $\widetilde{u}_i=u_i/ G_{x_0}$. Then,  $\widetilde{u}_i$ satisfies
\begin{align}
\label{eq u_i} \Delta_{\widetilde{g}} \widetilde{u}_{i}  &=
\frac14n(n-2) \widetilde{u}_{i}^{\frac{n+2}{n-2}} \quad\text{in }\,\Omega_i,\\
\label{eq-u i Boundary}\widetilde{u}_{i}&=\infty\quad\text{on }\partial \Omega_i,
\end{align}
and for any $m$,
\begin{equation}\label{u_i go to zero-a}
\widetilde{u}_i\rightarrow 0\quad\text{in }C^{m}_{\mathrm{loc}}( M \setminus\{x_0\})
\text{ as }i\rightarrow\infty.\end{equation}

Fix a point $p_0\in M \setminus\{x_0\}$. Then, $p_0 \in \Omega_i$, for $i$ sufficiently large.
Set $\widetilde{w}_i={\widetilde{u}_i}/{\widetilde{u}_i(p_0)}$. Then, $\widetilde{w}_i(p_0)=1$,
and $\widetilde{w}_i$ satisfies
\begin{align}
\label{eq w_i} \Delta_{\widetilde{g}} \widetilde{w}_{i}  &=
\frac14n(n-2)u_i(p_0)^{\frac{4}{n-2}} \widetilde{w}_{i}^{\frac{n+2}{n-2}} \quad\text{in }\,\Omega_i,\\
\label{eq-w i Boundary}\widetilde{w}_{i}&=\infty\quad\text{on }\partial \Omega_i.
\end{align}
By interior estimates, there exists a positive function $\widetilde{w}\in \widetilde{M}$ such that, for any $m$,
$$\widetilde{w}_i\rightarrow \widetilde{w}\quad\text{in }C^{m}_{\mathrm{loc}}( \widetilde{M})
\text{ as }i\rightarrow\infty,$$
and
\begin{equation}\Delta_{\widetilde{g}} \widetilde{w}=0  \quad\text{in  } \widetilde{M}.\end{equation}
Hence,
\begin{equation}L_g( G_{x_0}\widetilde{w})=0  \quad\text{in  } M \setminus\{x_0\}.\end{equation}
By the expansion of $G_{x_0}$ near $x_0$ and Proposition 9.1 in \cite{LiZhu1999}, we conclude that
$ \widetilde{w}$ converges to some constant as $x\to x_0$. Therefore,
$\widetilde{w}\equiv1$ in $\widetilde{M}$. Hence, for any $m$,
\begin{equation}\label{u_i behavioer2}\frac{u_i}{\widetilde{u}_i (p_0) G_{x_0}}\rightarrow 1
\quad\text{in }C^{m}_{\mathrm{loc}}( M \setminus\{x_0\})
\text{ as }i\rightarrow\infty.\end{equation}
In the following, we always discuss in the conformal normal coordinates near $x_0$.
Set
$$v_i= u_i^{-\frac{2}{n-2}}=
(\widetilde{u}_i (p_0))^{-\frac{2}{n-2}}
\Big(\frac{u_i}{\widetilde{u}_i (p_0) G_{x_0}}\Big)^{-\frac{2}{n-2}}G_{x_0}^{-\frac{2}{n-2}}.$$
We will fix a direction appropriately, which we call $x_1$.
Denote by $R^i_{11}$ the Ricci curvature of $g_i$
acting on the unit vector $v_i\frac{\partial}{\partial x_1}$ with respect to the metric $g_i$.
To study $R^i_{11}$ given by \eqref{Ricci cur  in v Manifold-1},
we need to analyze the expansion of $G_{x_0}^{-\frac{2}{n-2}}.$
See \cite{LeeParker1987} or \cite{SchoenYau1994} for details.

Now we discuss several cases.

\smallskip
{\it Case 3.1.} $n=3,4,5$, or $M$ is conformally flat in a neighborhood of $x_0$.
In this case, we have
$$G_{x_0}=r^{2-n}+A+O(r),$$
where $A$ is a constant. Since $\lambda(M,[g])< \lambda(S^n,[g_{S^n}])$, we have $A>0$ when $3\leq n \leq 7$,
or $M$ is locally conformally flat, or $M$ is spin. We also have $A>0$
when $M$ is just conformally flat in a neighborhood of $x_0$
under the assumption that the positive mass theorem holds.
Then,
\begin{align*}G_{x_0}^{-\frac{2}{n-2}}&=r^2-\frac{2}{n-2}Ar^n+O( r^{n+1}),\\
 \partial_rG_{x_0}^{-\frac{2}{n-2}}&=2r-\frac{2n}{n-2}Ar^{n-1}+O( r^{n}),\\
 \partial_{rr}G_{x_0}^{-\frac{2}{n-2}}&=2-\frac{2n(n-1)}{n-2}Ar^{n-2}+O( r^{n-1}),\end{align*}
and hence
\begin{align*}
G_{x_0}^{-\frac{2}{n-2}}  \partial_{rr}G_{x_0}^{-\frac{2}{n-2}}
-\frac12( \partial_r G_{x_0}^{-\frac{2}{n-2}})^2
=-2(n-1)Ar^n+ O( r^{n+1}).\end{align*}

For $n=3,4,5$,  by \cite{LeeParker1987} (Page 61),
$R_{11}(x_0)=0$, $R_{11,1}(x_0)=0$ and $R_{11,11}(x_0)\leq 0$, we have
$$R_{11}\leq C|x_1|^3\quad\text{on the $x_1$-axis near }x_0=0.$$
We also have $S_{g}(x_0)=0$ and $S_{g,1}(x_0)=0$, and hence
$$|S_g |\leq Cx_1^2\quad\text{on the $x_1$-axis near }x_0=0.$$
Take any $x_1>0$ small. Then, at the point $x_1 e_1$,
\begin{align}\label{eq-Ricci-order}v_i^2R_{11}\le C(\widetilde{u}_i(p_0))^{-\frac{4}{n-2}}x_1^7, \end{align}
and
\begin{align}\label{eq-scalar-order}v_i^2|S_g|\le C(\widetilde{u}_i(p_0))^{-\frac{4}{n-2}}x_1^6.\end{align}
For $i$ large, by \eqref{Ricci cur  in v Manifold-1}, we have,
at the point $x_1 e_1$,
\begin{align}\label{Ricci comp}
R^i_{11}\le (\widetilde{u}_i(p_0))^{-\frac{4}{n-2}}\big[-2(n-1)(n-2)Ax_1^n+Cx_1^6+o(1)\big],
\end{align}
where $o(1)$ denotes terms converging to zero as $i\to \infty$, uniformly for small $x_1$ away from 0.
The dominant term in \eqref{Ricci comp} is the $x_1^n$-term, with a negative coefficient. Hence,
the expression inside the bracket in \eqref{Ricci comp} is strictly less than 0,
for a fixed small $x_1\neq0$ and $i$ large.
Therefore, at the point $x_1e_1$, $R^i_{11}$
diverges to $ -\infty$ as $i\rightarrow\infty$. Hence, some component of the Ricci curvature of $g_i$
at the point $x_1e_1$ diverges to $\infty$ as $i\rightarrow\infty$.

If $M$ is conformally flat in a neighborhood of $x_0$,  then $R_{11}=0$ and $S_{g}=0$
on the $x_1$-axis and near $x_0=0$.
The $x_1^6$-term in  \eqref{Ricci comp} is absent.
Similarly, at the point $x_1e_1$ for $x_1>0$ sufficiently small,
$R^i_{11}$
diverges to $ -\infty$ as $i\rightarrow\infty$.

If we denote by $R^i_{rr}$ the Ricci curvature of $g_i$
acting on the unit vector $v_i\frac{\partial}{\partial r}$ with respect to the metric $g_i$, then
we conclude similarly that
$R^i_{rr}$ at $x$
diverges to $ -\infty$ as $i\rightarrow\infty$, for some $x$ sufficiently close to $x_0$.

\smallskip
{\it Case 3.2.} $n=6$ and $M$ is not conformally flat in a neighborhood of $x_0$.
In this case,
$$G_{x_0}(x)=r^{2-n}-\frac{n-2}{1152(n-1)} |W(x_0)|^2\log r -\frac{1}{96}S_{g,ij}(x_0)\frac{x^{i}x^{j}}{r^{2}}
+P(x)\log r +\alpha(x),$$
where $W$ is the Weyl tensor, $P(x)$ is a polynomial with $P(0)=0$, and $\alpha$ is a $C^{2,\mu}$-function.
We note that $W_{ijkl}$ is given by
\begin{align}\label{Weyl tensor}\begin{split}
W_{ijkl}  &=R_{ijkl}-\frac{1}{n-2}\big(R_{ik}g_{jl}-R_{il}g_{jk}+R_{jl}g_{ik}-R_{jk}g_{il}\big)\\
&\qquad+\frac{S_{g}}{(n-1)(n-2)}\big(g_{ik}g_{jl}-g_{il}g_{jk}\big).
\end{split}\end{align}

{\it Case 3.2.1.} If $W(x_0) \neq 0$,
then,
\begin{align*}G_{x_0}^{-\frac{2}{n-2}}&=r^2+\frac{1}{2880}|W(x_0)|^2r^6\log r + O(r^7\log r) ,\\
\partial_rG_{x_0}^{-\frac{2}{n-2}}&=2r+\frac{1}{480}|W(x_0)|^2r^5\log r+O( r^{5}),\\
\partial_{rr}G_{x_0}^{-\frac{2}{n-2}}&=2+\frac{1}{96}|W(x_0)|^2r^4\log r+O( r^{4}),\end{align*}
and hence
\begin{align*}
G_{x_0}^{-\frac{2}{n-2}}  \partial_{rr}G_{x_0}^{-\frac{2}{n-2}}
-\frac12( \partial_r G_{x_0}^{-\frac{2}{n-2}})^2
=\frac{1}{144}|W(x_0)|^2r^6 \log r+ O( r^{6}).\end{align*}
Take any $x_1>0$ small. Then, at the point $x_1 e_1$,
\eqref{eq-Ricci-order} and \eqref{eq-scalar-order} still hold.
For $i$ large, instead of  \eqref{Ricci comp}, we have,
at the point $x_1 e_1$,
\begin{align*}
R^i_{11}\le (\widetilde{u}_i(p_0))^{-1}\big[\frac{1}{36}|W(x_0)|^2x_1^6 \log x_1+Cx_1^6+o(1)\big].
\end{align*}
Similarly as in Case 3.1, at the point $x_1e_1$ for $x_1>0$ sufficiently small,
$R^i_{11}$
diverges to $ -\infty$ as $i\rightarrow\infty$.

Similarly, $R^i_{rr}$ at $x$
diverges to $ -\infty$ as $i\rightarrow\infty$, for some $x$ sufficiently close to $x_0$.
\smallskip

{\it Case 3.2.2.} We now consider the case $W(x_0)=0$. By \eqref{Weyl tensor},
we have $R_{ijkl}(x_0)=0$. Hence, $(\widetilde{M}, \widetilde{g})$ is asymptotically flat
of order 3.
Using the spherical coordinates, we set
$$\phi_4(\theta)=\frac{1}{96}S_{g,ij}(x_0)\frac{x^{i}x^{j}}{r^{2}},$$
and denote by $r^2g_2(\theta )$ the degree two part of  the Taylor expansion of $S_{g}$ at $x_0$.
Since
$$\sum_{i=1}^{n}S_{g,ii}(x_0)=-\frac{1}{6}|W|^2(x_0)=0,$$
then, $$\int_{S^{n-1}}g_2(\theta)d\theta=0.$$
By the positive mass theorem, see \cite{LeeParker1987} (Page 79, 80), we have
$$\int_{S^{n-1}}\big(\phi_4(\theta)+\alpha(0)\big)d\theta>0.$$
Along a radial geodesic $\{(r,\theta):\,0\le r\le \delta\}$, for a small constant $\delta$, we have
\begin{align*}G_{x_0}^{-\frac{2}{n-2}}&=r^2-\frac{1}{2}\big(\phi_4(\theta)+\alpha(0) \big) r^6 +o( r^{6}),\\
 \partial_rG_{x_0}^{-\frac{2}{n-2}}&=2r-3\big(\phi_4(\theta)+\alpha(0) \big) r^{5}+o( r^{5}),\\
 \partial_{rr}G_{x_0}^{-\frac{2}{n-2}}&=2-15\big(\phi_4(\theta)+\alpha(0) \big) r^{4}+o( r^{4}),\end{align*}
and hence
\begin{align*}
G_{x_0}^{-\frac{2}{n-2}}  \partial_{rr}G_{x_0}^{-\frac{2}{n-2}}
-\frac12( \partial_r G_{x_0}^{-\frac{2}{n-2}})^2
=-10\big(\phi_4(\theta)+\alpha(0) \big)  r^n+ o( r^{n}).\end{align*}
By \cite{LeeParker1987} (Page 61), along the radial geodesic $(\cdot,\theta)$,
$\partial_r^3R_{rr}(x_0)=0$ and $\partial_r^4R_{rr}(x_0)\leq 0$. Therefore,
for $i$ large, by \eqref{Ricci cur  in v Manifold-1}, we have,
along the radial geodesic $(\cdot,\theta)$,
\begin{align*}
R^i_{rr}|_{(r,\theta)}\le (\widetilde{u}_i(p_0))^{-1}\big[-\frac{1}{10}g_2(\theta)r^6
-40\big(\phi_4(\theta)+\alpha(0) \big)r^6 +o(r^6)+o(1)\big],
\end{align*}
where $o(1)$ denotes terms converging to zero as $i\to \infty$, uniformly for small $x$ away from 0. Hence,
\begin{align*}
\int_{S^{n-1}}R^i_{rr}d\theta\le (\widetilde{u}_i(p_0))^{-1}
\Big[-40r^6\int_{S^{n-1}}\big(\phi_4(\theta)+\alpha(0) \big)d\theta
+o(r^6)+o(1)\Big].
\end{align*}
Therefore, we can find $\theta_0\in S^{n-1}$ that
$R^i_{rr}$ at $(r,\theta_0)$
diverges to $ -\infty$ as $i\rightarrow\infty$, for some $r$ sufficiently small.
\smallskip

{\it Case 3.3.} $n\geq7$ and $M$ is not conformally flat in a neighborhood of $x_0$.
In this case,
\begin{align*}G_{x_0}(x)=r^{2-n}\Big[1+\sum_{i=4}^{n}\psi_{i}\Big]+c\log r+P(x)\log r+\alpha(x),\end{align*}
where $\psi_{i}$ is a homogeneous polynomial of degree $i$, $c$ is a constant,
$P(x)$ is a polynomial with $P(0)=0$, and $\alpha$ is a  $C^{2,\mu}$-function. We note that
$c=0$ and $P\equiv 0$
if $n$ is odd.
Moreover,
$$\psi_4(x)=\frac{n-2}{48(n-1)(n-4)}\Big(\frac{r^4}{12(n-6)} |W(x_0)|^2-S_{g,ij}(x_0)x^i x^jr^2\Big),$$
where $W$ is the Weyl tensor.

{\it Case 3.3.1.} First, we consider the case $|W(x_0)| \neq 0$. Note that $S_{g}(x_0)=0$, $\nabla_gS_{g}(x_0)=0$,
and $$\Delta_{g}S_{g}(x_0)=-\frac{1}{6}|W(x_0)|^2.$$
Without loss of generality, we assume $S_{g,11}(x_0)<0$.
Take any $x_1>0$ small. Then, at the point $x_1 e_1$,
\eqref{eq-Ricci-order} still holds.
Set
$$A=\frac{n-2}{48(n-1)(n-4)}\Big[\frac{1}{12(n-6)} |W(x_0)|^2
-S_{g,11}(x_0)\Big].$$
Then, on the positive $x_1$-axis near $x_0=0$, we have
\begin{align*}G_{x_0}^{-\frac{2}{n-2}}&=x_1^2-\frac{2}{n-2}Ax_1^6 +O(x_1^7) ,\\
 \partial_{x_1}G_{x_0}^{-\frac{2}{n-2}}&=2x_1-\frac{12}{n-2}Ax_1^5 +O(x_1^6),\\
 \partial_{x_1x_1}G_{x_0}^{-\frac{2}{n-2}}&=2-\frac{60}{n-2}Ax_1^4 +O(x_1^5) ,\end{align*}
and
\begin{align*}
&(n-2)[G_{x_0}^{-\frac{2}{n-2}}  \partial_{x_1x_1}G_{x_0}^{-\frac{2}{n-2}}
-\frac12( \partial_{x_1} G_{x_0}^{-\frac{2}{n-2}})^2]-\frac{1}{2(n-1)}S_{g} G_{x_0}^{-\frac{4}{n-2}}\\
&\quad  =-40Ax_1^6
-\frac{1}{4(n-1)}S_{g,11}(x_0)x_1^6+O(x_1^7) .\end{align*}
By the definition of $A$, we obtain
\begin{align*}
&(n-2)[G_{x_0}^{-\frac{2}{n-2}}  \partial_{x_1x_1}G_{x_0}^{-\frac{2}{n-2}}
-\frac12( \partial_{x_1} G_{x_0}^{-\frac{2}{n-2}})^2]-\frac{1}{2(n-1)}S_{g} G_{x_0}^{-\frac{4}{n-2}}\\
&\quad  =-\frac{1}{12(n-1)(n-4)}\Big[\frac{5(n-2)}{6(n-6)}|W(x_0)|^2
-(7n-8)S_{g,11}(x_0)\Big]x_1^6+O(x_1^7).\end{align*}
For $i$ large, instead of  \eqref{Ricci comp}, we have,
at the point $x_1 e_1$,
\begin{align*}
R^i_{11}\le (\widetilde{u}_i(p_0))^{-\frac{4}{n-2}}\big[-Bx_1^6+Cx_1^7+o(1)\big],
\end{align*}
for some positive constant $B$.
Then, we conclude
${R}_{11}^i$
at the point $x_1e_1$
diverges to $ -\infty$ as $i\rightarrow\infty$, for $x_1>0$ sufficiently small.

{\it Case 3.3.2.} We now consider the case $W(x_0)=0$. By \eqref{Weyl tensor}, we have $R_{ijkl}(x_0)=0$.
Using the spherical coordinates, we set
$$\psi_i=r^i\phi_i(\theta),$$
and denote by $r^ig_i(\theta )$ the  $i$-th Taylor expansion of $S_{g}$ at $x_0$.
Let $r^lg_l(\theta )$ be the first nonzero term in the Taylor expansion of $S_{g}$ at $x_0$.

{\it Subcase 3.3.2(a).} $2\leq l \leq n-5$.
By \cite{LeeParker1987} or \cite{SchoenYau1994}, we have $\psi_{i}=0$, $i=4,...,l-1$, and
$$\mathcal{L}\psi_{l+2} =-\frac{n-2}{4(n-1)} r^{l+2}g_l(\theta), $$
where
$$\mathcal{L}=-r^2\Delta+2(n-2)r\partial_r.$$
Here, $\Delta$ is the standard Laplacian on the Euclidean space, i.e.,
$$\Delta=\frac{\partial^2}{\partial r^2}+\frac{n-1}{r}\frac{\partial}{\partial r}+\frac{1}{r^2}\Delta_{S^{n-1}}.$$
Then, we have
 $$(l+2)(n-4-l)\int_{S^{n-1}}\psi_{l+2}=\int_{S^{n-1}}\mathcal{L}\psi_{l+2} =-\frac{n-2}{4(n-1)} \int_{S^{n-1}}r^{l+2}g_l(\theta).$$
Hence,
\begin{equation}\label{psi-g}\int_{S^{n-1}}\phi_{l+2}=-\frac{n-2}{4(n-1)(l+2)(n-4-l)}\int_{S^{n-1}}g_l(\theta).\end{equation}
We also have
\begin{align}\label{intergral indentity1}\begin{split}
\int_{S^{n-1}}g_l(\theta)&=\frac{r^{2-n-l}}{l}\int_{0}^{r}\int_{S^{n-1}}\Delta\big(s^l g_l(\theta)\big)s^{n-1}dr d\theta \\
&=\lim_{r\rightarrow0}\frac{r^{2-n-l}}{l}\int_{B_{r}(x_0)}\Delta_{g}S_{g}dV_{g}\\
&=\lim_{r\rightarrow0}\frac{-r^{2-n-l}}{6l}\int_{B_{r}(x_0)}|W|^2dV_{g}\leq 0.
\end{split}\end{align}
Note that the sign of $\lim_{r\rightarrow0}r^{2-n-l}\int_{B_{r}(x_0)}|W|^2dV_{g}$ is independent of $g\in[g]$.
Hence, if for some $ i$ with $2\leq i \leq n-5$,
$$0<\lim_{r\rightarrow0}r^{2-n-i}\int_{B_{r}(x_0)}|W|^2dV_{g}<\infty,$$
then the $i$-th Taylor expansion of $S_{g}$ at $x_0$ must not be identical to zero.

Note
\begin{align*}G_{x_0}^{-\frac{2}{n-2}}&=r^2-\frac{2}{n-2} r^{l+4}\phi_{l+2} +o(  r^{l+4}),\\
 \partial_rG_{x_0}^{-\frac{2}{n-2}}&=2r-\frac{2}{n-2}(l+4) r^{l+3}\phi_{l+2} +o(  r^{l+3}),\\
 \partial_{rr}G_{x_0}^{-\frac{2}{n-2}}&=2-\frac{2}{n-2}(l+4)(l+3)r^{l+2}\phi_{l+2}+o(  r^{l+2}),\end{align*}
and hence,
\begin{align*}
&(n-2)[G_{x_0}^{-\frac{2}{n-2}}  \partial_{x_1x_1}G_{x_0}^{-\frac{2}{n-2}}
-\frac12( \partial_{x_1} G_{x_0}^{-\frac{2}{n-2}})^2]-\frac{1}{2(n-1)}S_{g} G_{x_0}^{-\frac{4}{n-2}}\\
&\quad  =\Big[-2(l+2)(l+3)\phi_{l+2}(\theta)-\frac{1}{2(n-1)}g_{l}(\theta) \Big]r^{l+4}+o(r^{l+4}).\end{align*}

By \cite{LeeParker1987} (Page 61), along a radial geodesic $(\cdot,\theta)$,
$\partial^i_rR_{rr}(x_0)=0$, $i=1,...,l-1$, and $\partial^l_rR_{rr}(x_0)\leq 0$.
Therefore,
for $i$ large, by \eqref{Ricci cur  in v Manifold-1}, we have,
along a radial geodesic $(\cdot,\theta)$,
\begin{align}\label{Ricci comp3}\begin{split}
R^i_{rr}|_{(r,\theta)}&\le (\widetilde{u}_i(p_0))^{-\frac{4}{n-2}}\Big\{ \Big[-2(l+2)(l+3)\phi_{l+2} (\theta)
-\frac{1}{2(n-1)}g_{l}(\theta) \Big]r^{l+4}\\
&\qquad+o(r^{l+4})+o(1)\Big\},
\end{split}\end{align}
where $o(1)$ denotes terms converging to zero as $i\to \infty$, uniformly for small $x$ away from 0.
By \eqref{psi-g} and \eqref{intergral indentity1}, we have
\begin{align*}&-2(l+2)(l+3)\int_{S^{n-1}}\phi_{l+2}(\theta)d\theta-\frac{1}{2(n-1)}\int_{S^{n-1}}g_{l}(\theta)d\theta\\
&\qquad=\frac{(n-2)(l+3)-(n-4-l)}{2(n-1)(n-4-l)}\int_{S^{n-1}}g_l(\theta) d\theta
\leq 0.
\end{align*}
By \cite{LeeParker1987} Lemma 5.3 or \cite{SchoenYau1994} chapter 5,
$$2(l+2)(l+3)\phi_{l+2} \neq \frac{1}{2(n-1)}g_{l}.$$
Hence, we can find $\theta_0 \in S^{n-1}$ such that
$$-2(l+2)(l+3)\phi_{l+2} (\theta_0)-\frac{1}{2(n-1)}g_{l}(\theta_0) \leq -\epsilon_0,$$
for some positive constant $\epsilon_0$.
Therefore, along the radial geodesic $(r,\theta_0)$,
\begin{align*}
R^i_{rr}|_{(r,\theta_0)}\le& (\widetilde{u}_i(p_0))^{-\frac{4}{n-2}}\Big[-\epsilon_0r^{l+4}+o(r^{l+4})+o(1)\Big].
\end{align*}
Then, we conclude
${R}_{rr}^i$
at the point $(r,\theta_0)$
diverges to $ -\infty$ as $i\rightarrow\infty$, for $r>0$ sufficiently small.
\smallskip

{\it Subase 3.3.2(b).} $l\geq n-4$.
When $n$ is even, we have
$$\mathcal{L}(\psi_{n-2}+cr^{n-2} \log r)=\mathcal{L}\psi_{n-2}-(n-2)cr^{n-2}  =-\frac{n-2}{4(n-1)} r^{n-2}g_{n-4}(\theta).$$
Then,
\begin{align}\label{intergral indentity2}\begin{split}
(n-2)c w_{n-1} &=-r^{2-n}\int_{S^{n-1}}\mathcal{L}(\psi_{n-2}+cr^{n-2} \log r)\\
&=\frac{(n-2)r^{6-2n}}{4(n-1)(n-4)}\int_{0}^{r}\int_{S^{n-1}}\Delta\big(s^{n-4}g_l(\theta)\big)s^{n-1}dr d\theta \\
&=\lim_{r\rightarrow0}\frac{(n-2)r^{6-2n}}{4(n-1)(n-4)}\int_{B_{r}(x_0)}\Delta_{g}S_{g}dV_{g}\\
&=\lim_{r\rightarrow0}\frac{-(n-2)r^{6-2n}}{24(n-1)(n-4)}\int_{B_{r}(x_0)}|W|^2dV_{g}\leq 0.
\end{split}\end{align}
We note $$\lim_{r\rightarrow0}r^{6-2n}\int_{B_{r}(x_0)}|W|^2dV_{g}=0,$$
when $n$ is odd, since $\int_{S^{n-1}}g_{n-2}(\theta) d\theta=0$ when $n$ is odd.

If $c>0$, we can proceed as the proof of
Case 3.2.1, $n=6$ and $|W(x_0)|\neq 0$, and conclude that
$R^i_{rr}$ at $x$
diverges to $ -\infty$ as $i\rightarrow\infty$, for some $x$ sufficiently close to $x_0$.

In general, we first consider the case that
there exist a pair $(i,j)\in\{1\cdots n\}\times\{1\cdots n\}$ and a constant $k<[\frac{n-4}{2}]$
such that $R_{ij}\neq0$ and $k$ is the order of the first nonzero term in the Taylor expansion of $R_{ij}$ at $x_0$.

Without loss of generality, we assume the order of the first nonzero term in the Taylor expansion
of some $ R_{pq}$ at $x_0$ is $k$, $k<[\frac{n-4}{2}]$, and all other $R_{ij}$
vanish up to order $k$ at $x_0$.  Then, by \eqref{Weyl tensor}, all $R_{ijkl}$ vanish up to order $k$, and hence,
all $g_{ij}-\delta_{ij}$ vanish up to order $k+2$. By a rotation, we can assume
$$\frac{\partial^k}{\partial x_1^k}R_{pq}|_{x_0}\neq0.$$
By \cite{LeeParker1987} (Page 61),  $(p,q)\neq (1,1)$.
If $p \neq1$ and $q \neq 1$, by a rotation, we can assume $p=q=2$. Otherwise, we can assume $(p,q)=(1,2)$.

We consider the case $(p,q)=(2,2)$. By the Gauss Lemma, we have
$$x_j=\sum_{i=1}^n g_{ji}x_i.$$
Then, we have, on the $x_1$-axis near $x_0=0$,
\begin{align}\label{g12g22}1&=x_1\frac{\partial}{\partial x_2 } g_{12}+g_{22},\\
\label{g12g11}0&=x_1\frac{\partial}{\partial x_2 } g_{11}+g_{12},\end{align}
and
$$0=x_1\frac{\partial^2 }{\partial x_2^2}g_{22}+2\frac{\partial }{\partial x_2}g_{12}.$$
We also have, at $x_0=0$,
$$(k+2)\frac{\partial^{k+2} }{\partial x_1^{k+1} \partial x_2}g_{21}+\frac{\partial^{k+2} }{\partial x_1^{k+2} }g_{22}=0,$$
and
$$(k+1)\frac{\partial^{k+2} }{\partial x_1^{k} \partial^2 x_2}g_{11}
+2\frac{\partial^{k+2} }{ \partial x_1^{k+1} \partial x_2}g_{12}=0.$$
Hence, we have, at $x_0$,
\begin{align*}
\frac{\partial^{k} }{\partial x_1^{k} }R_{1212}&=\frac{1}{2}
\Big( 2\frac{\partial^{k+2} }{\partial x_1^{k+1} \partial x_2}g_{21}- \frac{\partial^{k+2} }{\partial x_1^{k+2} }g_{22}
- \frac{\partial^{k+2}}{\partial x_1^{k} \partial^2 x_2}g_{11}\Big)\\
&=-\frac{k+3}{k+1}\frac{\partial^{k+2} }{\partial x_1^{k+2} }g_{22}(x_0).
\end{align*}
Therefore, by \eqref{Weyl tensor}, we have
\begin{equation}\label{r22 expansion}\frac{\partial^{k} }{\partial x_1^{k} }|_{x_0}R_{22}
=(n-2)\frac{\partial^{k} }{\partial x_1^{k} }|_{x_0}R_{1212}
=-(n-2)\frac{k+3}{k+1}\frac{\partial^{k+2} }{\partial x_1^{k+2} }g_{22}(x_0)<0.\end{equation}
Then, for $i$ large, by \eqref{Ricci cur  in v Manifold-1}, we have,
at the point $x_1 e_1$,
\begin{align*}
R^i_{22}= (\widetilde{u}_i(p_0))^{-\frac{4}{n-2}}&\Big[\frac{1}{k!}\frac{\partial^{k}}{\partial x_1^{k} }|_{x_0}R_{22}x_1^{k+4}
+(n-2)\Big(-2\Gamma_{22}^1 x_1^3-2(g_{22}-1)x_1^2\Big)\\
&\quad+O(x_1^{k+5} )+o(1)\Big],
\end{align*}
where $o(1)$ denotes terms converging to zero as $i\to \infty$, uniformly for small $x_1$ away from 0.
At the point $x_1 e_1$,
$$\Gamma_{22}^1 =\frac{1}{2}\Big(2\frac{\partial}{\partial x_2}g_{12}-\frac{\partial}{\partial x_1}g_{22}\Big).$$
Combining with \eqref{g12g22}, we get, at the point $x_1 e_1$,
\begin{align*}
R^i_{22}&= (\widetilde{u}_i(p_0))^{-\frac{4}{n-2}}\Big[\frac{1}{k!}\frac{\partial^{k}}{\partial x_1^{k} }|_{x_0}R_{22}x_1^{k+4}
+(n-2)\frac{\partial}{\partial x_1}g_{22}x_1^3+O(x_1^{k+5} )+o(1)\Big]\\
&=(\widetilde{u}_i(p_0))^{-\frac{4}{n-2}}\Big[-(n-2)\frac{k+2}{(k+1)!}
\frac{\partial^{k+2} }{\partial x_1^{k+2} }g_{22}(x_0)x_1^{k+4}+O(x_1^{k+5} )+o(1)\Big].
\end{align*}
Then, we conclude
${R}_{22}^i$
at the point $x_1e_1$
diverges to $ -\infty$ as $i\rightarrow\infty$, for $x_1>0$ sufficiently small.

If $(p,q)=(1,2)$, we can argue similarly to conclude that
$|{R}_{12}^i|$ at the point $x_1e_1$
diverges to $ \infty$ as $i\rightarrow\infty$, for $x_1>0$ sufficiently small.

\smallskip
We now consider the case that the order of the first nonzero term in the Taylor expansion
of all $R_{pq}$ is greater or equal to $[\frac{n-4}{2}]$ at $x_0$, and
 $$\lim_{r\rightarrow0}r^{6-2n}\int_{B_{r}(x_0)}|W|^2dV_{g}=0.$$
Then, by \eqref{Weyl tensor}, the order of the first nonzero term in the Taylor expansion of
$R_{ijkl}$ at $x_0$ is greater or equal to $[\frac{n-4}{2}]$,
and hence,  the order of the first nonzero term in the Taylor expansion of $g_{ij}-\delta_{ij}$
at $x_0$ is greater or equal to $[\frac{n-4}{2}]+2$.
Hence, $(\widetilde{M}, \widetilde{g})$ is asymptotically flat
of order $[\frac{n-4}{2}]+2$. Thus, the ADM-mass of $(\widetilde{M}, \widetilde{g})$ is well defined.
By the positive mass theorem, we have $$\int_{S^{n-1}}\big(\phi_{n-2}(\theta)+\alpha(0)\big)d\theta>0.$$
Then,  we can proceed as in the proof of
Case 2.2, $n=6$ and $|W(x_0)|= 0$, and find $\theta_0\in S^{n-1}$ such that
$R^i_{rr}$ at $(r,\theta_0)$
diverges to $ -\infty$ as $i\rightarrow\infty$, for some $r$ sufficiently small.
\end{proof}

\begin{remark}\label{unknown} We point out that we used the positive mass theorem in the proof of
Theorem \ref{blow up-one point} if the Yamabe invariant of $(M,[g])$ is
between zero and that of the standard sphere and one of the following conditions holds:
(1) $M$
is locally conformally flat, (2) $3\le n\le 5$, or (3) for $n\ge 6$, the Weyl tensor $W$ at $x_0$ satisfies
$$\nabla^i|W|^2(x_0)= 0\quad\text{for any } i=0, \cdots, n-6.$$
\end{remark}

\begin{remark}\label{rmk-single-point-sphere}
The blow-up phenomena in Theorem \ref{blow up-one point} are
significantly different from those for the case that the underlying manifold is $S^n$.
For example, take $\Omega=S^n\setminus B_{r}(e_n)$, where $B_{r}(e_n)$
is a small ball on $S^n$ centered at the north pole.  Then,
$\Omega$ is close to $ S^n\backslash \{e_n\}$
and the complete conformal metric $g_\Omega$ in $\Omega$ with the constant scalar curvature
$-n(n-1)$ has a constant sectional curvature $-1$! This can be verified
by the stereographic projection, as the image of $S^n\setminus B_{r}(e_n)$ under
the stereographic projection from the north pole is a ball in $\mathbb R^n$ centered at the origin.
\end{remark}

We note that Theorem \ref{them-rigidity} follows easily from
Theorem \ref{blow up-one point}.
Now we are ready to prove Theorem \ref{thrm-large-positive-Ricci}.

\begin{proof}[Proof of Theorem \ref{thrm-large-positive-Ricci}]
Let  $u_i$ be the solution
of \eqref{eq-MEq} and \eqref{eq-MBoundary} in $\Omega_i$.
Then,
$g_{i}=u_i^{\frac{4}{n-2}}g.$

The proof of Theorem \ref{prop-general mainfold}  can be adapted to prove Case 1, i.e.,
$\Gamma$ contains a submanifold of dimension $j$, with $1\le j\leq \frac{n-2}{2}$.

Next, we consider Case 2, i.e., $(M, g)$  is not conformally equivalent to the standard sphere $S^n$
and $\Gamma$ consists of finitely many points.
If $\lambda(M,[g])\leq 0$,
the proof of Theorem \ref{prop-general mainfold} and
Theorem \ref{blow up-one point} can be adapted
to yield the desired conclusion.
Hence, we only need to discuss the case $\lambda(M,[g])> 0$ and
$\Gamma$ consists of finitely many points $\{p_1, \cdots, p_k\}$.

Let $G_{p_j}\in C^{\infty}(M \setminus\
\{p_j\})$ be  the Green's function for the conformal Laplacian $L_g$
with the pole at $p_j$, $j=1, \cdots, k$, respectively; namely,
\begin{equation*}L_gG_{p_j}=(n-2)\omega_{n-1}\delta_{p_j},\,\, G_{p_j}>0,\end{equation*}
where $\omega_{n-1}$ is the volume of $S^{n-1}$.
Up to conformal factors, we assume $(M, g)$ has conformal normal coordinates in
small neighborhoods of $p_i$.
Consider the metric
$$\widetilde{g}
 = \big(G_{p_1}+\cdots+G_{p_k}\big)^{\frac{4}{n-2}}g\quad\text{on }\widetilde{M} =M \setminus\
\{p_1, \cdots, p_k\}.$$
Then,  $(\widetilde{M}, \widetilde{g})$ is an asymptotically flat and scalar flat manifold,
and $\widetilde{g}$ has an asymptotic expansion near infinity.

Set $u_i= (G_{p_1}+\cdots+G_{p_k})\widetilde{u}_i$. Then,  $\widetilde{u}_i$ satisfies
\begin{align*}
\Delta_{\widetilde{g}} \widetilde{u}_{i}  &=
\frac14n(n-2) \widetilde{u}_{i}^{\frac{n+2}{n-2}} \quad\text{in }\,\Omega_i,\\
\widetilde{u}_{i}&=\infty\quad\text{on }\partial \Omega_i,
\end{align*}
and, for any $m$,
\begin{equation*}
\widetilde{u}_i\rightarrow 0\quad\text{in }C^{m}_{\mathrm{loc}}( M \setminus\{p_1, \cdots, p_k\})
\text{ as }i\rightarrow\infty.\end{equation*}
Fix a point $p_0\in M \setminus\{p_1, \cdots, p_k\}$. Then, for $i$ sufficiently large, $p_0 \in \Omega_i$.
Set $\widetilde{w}_i={\widetilde{u}_i}/{\widetilde{u}_i(p_0)}$.
Then, $\widetilde{w}_i(p_0)=1$ and $\widetilde{w}_i$ satisfies
\begin{align*}
\Delta_{\widetilde{g}} \widetilde{w}_{i}
&= \frac14n(n-2)u_i(p_0)^{\frac{4}{n-2}} \widetilde{w}_{i}^{\frac{n+2}{n-2}} \quad\text{in }\,\Omega_i,\\
\widetilde{w}_{i}&=\infty\quad\text{on }\partial \Omega_i.
\end{align*}
By interior estimates, there exists a positive function $\widetilde{w}\in \widetilde{M}$ such that, for any $m$,
$$\widetilde{w}_i\rightarrow \widetilde{w}\quad\text{in }C^{m}_{\mathrm{loc}}( \widetilde{M})
\text{ as }i\rightarrow\infty,$$
and
\begin{equation*}\Delta_{\widetilde{g}} \widetilde{w}=0  \quad\text{in  } \widetilde{M}.\end{equation*}
Hence,
\begin{equation*}L_g\big( (G_{p_1}+\cdots+G_{p_k})\widetilde{w}\big)=0
\quad\text{in  } M \setminus\{p_1, \cdots, p_k\}.\end{equation*}
By the expansions of $G_{p_j}$ near $p_j$, $j=1, \cdots, k$,
respectively, and Proposition 9.1 in \cite{LiZhu1999}, we conclude that
$ \widetilde{w}$ converges to some constant $\alpha_j$ as $x\to p_j$.
Without loss of generality, we assume $$\alpha_1  \geq\alpha_2\geq\cdots\geq\alpha_k\geq0.$$
Then, $\alpha_1\ge 1$.
By Proposition 9.1 in \cite{LiZhu1999},
$(G_{p_1}+\cdots+G_{p_k})\widetilde{w}$ can be extended to a $C^2$-function in a neighborhood of $p_j$ if $\alpha_j=0$.

If some of $\alpha_1$, $\cdots$, $\alpha_k$ is zero, we denote by  $l$
the first integer in $\{1, \cdots, k\}$ such that $\alpha_l=0$.  Otherwise, we set $l=k+1$.
We always have $l\geq 2$.

Consider the metric
$$\widehat{g}
 = (G_{p_1}+\cdots+G_{p_{l-1}})^{\frac{4}{n-2}}g\quad\text{on }\widehat{M} =M \setminus\
\{p_1, \cdots, p_{l-1}\}.$$
Set $(G_{p_1}+\cdots+G_{p_k})\widetilde{w}= (G_{p_1}+\cdots+G_{p_{l-1}})\widehat{w}$.
Then, $\widehat{w}$ satisfies\begin{equation*}
L_g\big( (G_{p_1}+\cdots+G_{p_{l-1}})\widehat{w}\big)=0  \quad\text{in  } M \setminus\{p_1,\cdots,p_{l-1}\},\end{equation*}
and
\begin{equation*}\Delta_{\widehat{g}} \widehat{w}=0  \quad\text{in  } \widehat{M}.\end{equation*}
We also have that $ \widehat{w}$ converges to $\alpha_j$ as $x\to p_j$, $j=1, \cdots, l-1$.
By Proposition 9.1 in \cite{LiZhu1999} and the maximum principle, we have,  near the point $p_{l-1}$,
$$ \widehat{w}=\alpha_{l-1}+C_{l-1}r^{n-2}+O(r^{n-1}),$$
for some nonnegative constant $C_{l-1}$.
Then, the proof follows similarly as that of Theorem \ref{blow up-one point}.

\smallskip

Next, we consider Case 3, i.e., $(M,g)$ is conformally equivalent to the standard sphere $S^n$
and $\Gamma$ consists of at least two but only finitely many points.
We can assume $(M,g)=(S^n,g_{S^n})$.  By Lemma \ref{lemme-convergence},
we have, for any $m$,
$$u_i\rightarrow 0\quad\text{in }C^{m}_{\mathrm{loc}}(S^n \setminus \Gamma)
\text{ as }i\rightarrow\infty.$$
Set $v_i=u_i^{-\frac{2}{n-2}}$. Then,
$$v_i \text{ diverges to $\infty$ locally uniformly in $S^n \setminus\Gamma$
 as $i\rightarrow\infty$}.$$
Take two different points $p ,q \in \Gamma$ and let  $\sigma_{pq}$ be
the shorter geodesic connecting $p$ and $q$.
Up to a conformal transform if necessary,
we  assume $|\sigma_{pq}|=2\epsilon$, which is less than $\frac{1}{100n}$,
and $\sigma_{pq}\bigcap\Gamma=\{p,q\}$.
We parametrize $\sigma_{pq}$ by its arc length $t\in[0, 2\epsilon]$,
with $p$ corresponding to  $t=0$ and $q$ to $t=2\epsilon$.


For $i$ large,
let $p_i$ and $q_i$ be the points  parametrized by $\widehat{t}_i$ and $\widetilde{t}_i$,
respectively, where
\begin{align*}\widehat{t}_i&=\min \{t'\in[0, \epsilon] |te_n \in \Omega_i, \text{ for any } t\in ( t',\epsilon]\},\\
\widetilde{t}_i&=\max \{t'\in [\epsilon, 2\epsilon] |t' \in \Omega_i ,\text{ for any } t\in [\epsilon,t')\}.\end{align*}
Then, $p_i, q_i\in \partial\Omega_i$ and
$$p_i\rightarrow p,\quad q_i\rightarrow q.$$
For convenience, we denote by $v_i(t)$ the function $v_i$ restricted to the geodesic $\sigma_{pq}$.
By the polyhomogenous expansion of $v_i$, we have $|\partial_t v_i(p_i)|\leq 1$ and $|\partial_t v_i(q_i)|\leq 1$.

Since $v_i(\epsilon)\to \infty$, for $i$ large,
we take $t_i \in (\widehat{t}_i,\widetilde{t}_i)$ such that, for any $t\in (\widehat{t}_i,\widetilde{t}_i)$,
$$\partial_tv_i(t)\leq \partial_tv_i(t_i).$$
Then,
$$\partial_tv_i(t_i)>\frac{v_i(\epsilon)-0}{\epsilon-\widehat{t}_i}\geq\frac{ v_i(\epsilon)}{\epsilon},$$
and $$\partial_{tt}v_i( t_i)=0.$$
Denote by $R^i_{tt}$ the Ricci curvature of $g_i$
acting on the unit vector $v_i\frac{\partial}{\partial t}$ with respect to the metric $g_i$.
Then, we can verify at the point $t_ie$, $R^i_{tt}$
diverges to $ -\infty$ as $i\rightarrow\infty$.
\end{proof}

\section{General Domains in Euclidean Spaces}\label{sec-Examples}

In this section, we present several examples of smooth bounded domains $\Omega$ in the Euclidean space
and examine whether
the complete conformal metrics $g_\Omega$ associated with the Loewner-Nirenberg problem
have negative Ricci curvatures. We demonstrate by these examples
the complexity of the issue studied in this paper.
Topological conditions are not sufficient to determine whether
these complete conformal metrics have negative Ricci curvatures.

There are two classes of examples. First, we construct nonconvex smooth domains in which
the complete conformal metrics
with a constant scalar curvature still have negative Ricci curvatures.
Second, we construct bounded smooth domains
in which the complete conformal metrics have positive Ricci components at some points.

By the Cartan-Hadamard Theorem, we know that $\pi_{i}(\Omega)=0$ for $i\geq2$ if $g_\Omega$
has negative sectional curvatures in $\Omega$, where $\pi_{i}(\Omega)$ is the $i$-th homotopy group of $\Omega$.
The following example shows that $\pi_{1}(\Omega)=0$ is not necessary for $g_\Omega$
to have negative sectional curvatures.

\begin{example}\label{exa-nonzero-pi1} Set
$$\Omega_r=\Big\{(x_1, \cdots, x_n)|\Big(x_1-\frac{x_1}{\sqrt{x_1^2+x_n^2}}\Big)^2+\sum_{i=2}^{n-1}x_i^2
+ \Big(x_n-\frac{x_n}{\sqrt{x_1^2+x_n^2}}\Big)^2 < r^2 \Big\}\subset\mathbb R^n,$$
where $0<r<{1}/{100}$. Then,  $\pi_{1}(\Omega_r)= \mathbb Z$ and  $\pi_{i}(\Omega)=0$ for $i\geq2$.
We claim that $g_{\Omega_r}$ has negative sectional curvatures in $\Omega_r$, if $r$ is sufficiently small.
By symmetry, it suffices to show $g_{\Omega_r}$ has negative sectional curvatures in
$$\Big\{(x_1, \cdots, x_n)|(x_1-1)^2+\sum_{i=2}^{n-1}x_i^2<r^2, x_n=0 \Big\}\subset\mathbb R^n.$$
Note that
$\Omega_r$ is transformed to
\begin{align*}\widetilde{\Omega}_r
&=\Big\{y\in \mathbb R^n|\Big(y_1+\frac{1}{r}-\frac{ry_1+1}{r\sqrt{(ry_1+1)^2+(ry_n)^2}}\Big)^2\\
&\qquad+\sum_{i=2}^{n-1}y_i^2+ \Big(y_n-\frac{y_n}{\sqrt{(ry_1+1)^2+(ry_n)^2}}\Big)^2 < 1 \Big\},
\end{align*}
under the transform
$$y_1=\frac{x_1-1}{r},\, y_2=\frac{x_2}{r},\,\cdots,\,y_n=\frac{x_n}{r}.$$
Then, $\{ \widetilde{\Omega}_r\}$ converges  in $C^{k}$, for any integer $k\ge 1$, to
$$\Omega_0=\{(x_1, \cdots, x_n)|x_1^2+\cdots+x_{n-1}^2<1 \}\subset\mathbb R^n$$
in any compact sets in $\mathbb R^n$.
Let $u^0$ be the positive solution of \eqref{eq-MainEq}-\eqref{eq-MainBoundary}
for $\Omega=\Omega_0$. Then, $u^0$ has the form
$$u^0(x_1,...,x^n)=u^0(x_1,...,x_{n-1},0).$$
By the same method as in the proof of Theorem \ref{main reslut},
we can prove $(u^0)^{\frac{4}{n-2}}|dx|^2$ has negative sectional curvatures in $\Omega_0$.
Then, the polyhomogeneous expansions for $u_r$ imply that $g_{\widetilde{\Omega}_r}$
has negative sectional curvatures in
$$\{(x_1, \cdots, x_n)|1-\delta<x_1^2+\cdots+x_{n-1}^2<1, x_n=0 \},$$ for some small $\delta>0$,
independent of $r$.
It is straightforward to prove that $g_{\widetilde{\Omega}_r}$ has negative sectional curvatures in
$$\{(x_1, \cdots, x_n)|x_1^2+\cdots+x_{n-1}^2\leq 1-\delta, x_n=0 \},$$
since $\{\widetilde{\Omega}_r \}$ converges to $\Omega_0$
in any compact sets in $\mathbb R^n$.\end{example}

In the next example, we construct
a  bounded smooth domain $\Omega\subset\mathbb R^n$ which is diffeomorphic
to the unit ball and cannot be conformally  transformed to a bounded convex domain such that
the complete conformal metric $g_\Omega$ with a negative scalar curvature possesses
negative sectional curvatures in $\Omega$.

\begin{example}\label{exa-not-conformal-to-convex}
Let $\sigma$ be an $U$-shaped smooth curve in $\mathbb R^n$ with two endpoints $p$ and $q$.
Let $\Omega^{r}$, $0<r<{1}/{100}$,
be a family of tubular domains with smooth boundaries satisfying the following conditions:

(A1) $\Omega^{r_2}\subseteq\Omega^{r_1}$ if $r_2 <r_1$;

(A2) $\bigcap \Omega^{r} =\sigma$;

(A3) For a fixed point $x_0 \in \sigma\setminus \{p,q\}$, $\{x\in \mathbb R^n|rx+x_0\in \Omega^r\}$
converges  in $C^{k}$, for any integer $k\ge 1$,  to a smooth domain which is equal to
$\{(x_1, \cdots, x_n)|x_1^2+\cdots+x_{n-1}^2<1 \}\subset\mathbb R^n$ up to an Euclidean transformation
 in any compact set in $\mathbb R^n$ as $r\rightarrow0$;

(A4) At the point $p$ or $q$, $\{x\in \mathbb R^n|rx+x_0\in \Omega^r\}$ converges
in $C^{k}$, for any integer $k\ge 1$, to a smooth domain which is equal to
$\Omega^0$ up to an Euclidean transformation in any compact set in $\mathbb R^n$ as $r\rightarrow0$,
where $\Omega^0$ is a smooth convex domain which coincides
$\{(x_1, \cdots, x_n)|x_1^2+\cdots+x_{n-1}^2<1 \}$ when $x_n\geq0$, and coincides
$\{(x_1, \cdots, x_n)|-\sqrt{x_1^2+\cdots+x_{n-1}^2}<x_n<0\}$
when $-1<x_n<{1}/{3}$.

Then, $\Omega^{r}$ cannot be conformally transformed to  bounded convex domains if $r$ is sufficiently small.
Otherwise, there would exist an arc $\sigma_r$ with endpoints $p$ and $q$ such that
$\sigma_r \subseteq\Omega_r$.
Let $w$ be the positive solution of \eqref{eq-MainEq}-\eqref{eq-MainBoundary} for $\Omega=\Omega^0$.
Then by approximation and using the same method as in the proof of Theorem \ref{main reslut},
we can prove $w^{-\frac{2}{n-2}}$ is concave in $\Omega^0$ and $w^{\frac{4}{n-2}}|dx|^2$ has
negative sectional curvatures in $\Omega^0$. Arguing as in
Example \ref{exa-nonzero-pi1}, we can show that the sectional curvatures in $\Omega^0$ is close
to the sectional curvatures in $\Omega_0$ when $x_n$ is sufficiently large.
Hence, the sectional curvatures in $\Omega^0$ are bounded above by a negative constant.
Then, arguing as in Example \ref{exa-nonzero-pi1} again, we obtain that $g_{\Omega^r}$ has negative
sectional curvatures in $\Omega^r$ when $r$ is sufficiently small.
\end{example}

\begin{example}\label{negative ricci domain}
For $R>r>0$, consider the annular region $\Omega_{R,r}=B_{R}\setminus B_r$  in $\mathbb R^n$.
Let $g_{\Omega_{R,r}}$ be the complete conformal metric
with the constant scalar curvature $-n(n-1)$ in $\Omega_{R,r}$.
Arguing as in Example \ref{exa-nonzero-pi1},
we can prove that the maximum Ricci curvature component of  $g_{\Omega_{R,r}}$ is close
to $-{n}/{2}$ as $R$ and $r$ tend to 1. Hence, $g_{\Omega_{R,r}}$ has negative
Ricci curvatures in $\Omega_{R,r}$ as $R$ is sufficiently close to $r$.
\end{example}

In the rest of this section, we construct bounded domains
in which the complete conformal metrics have positive Ricci components at some points.
The most straightforward way to do this is to combine Theorem \ref{thrm-large-positive-Ricci}
for the case $(M,g)=(S^n,g_{S^n})$
and the stereographic projections.

We identify $\mathbb{R}^{n}$ in $\mathbb{R}^{n+1}$ as $\mathbb{R}^{n}\times \{0\}$
and write
$x=(x_1, \cdots, x_n)\in\mathbb R^n$.
Then,
$$S^n=\{(x,x_{n+1}):\,|x|^{2}+x_{n+1}^{2}=1\}.$$
Consider the transform  $T: \mathbb R^{n}\to S^n$ given by
$$T(x)=\Big(\frac{2 x}{1+|x|^2},\frac{|x|^2-1}{1+|x|^2}\Big).$$
Then, $T$ is the inverse transform of the stereographic projection
which lifts $\mathbb{R}^{n}\times \{0\}$ to $S^n$.

\begin{prop}\label{prop-general example} Let $\Gamma$ be a set in $S^n$ as in Theorem
\ref{thrm-large-positive-Ricci}, containing the north pole.
Suppose $\widetilde{\Omega}_i$ is a sequence of increasing smooth domains
in $S^n$ which converges to $S^{n}\setminus \Gamma$, with $\partial  \widetilde{\Omega}_i$
not containing the north pole, and set
$\Omega_i=T^{-1}(\widetilde{\Omega}_i)$.
Assume $g_i$ is the complete conformal metric in $\Omega_i$
with the constant scalar curvature $-n(n-1)$. Then,
for sufficiently large $i$, $g_i$
has a positive Ricci curvature component somewhere in
$\Omega_{i}$. Moreover,
the maximal Ricci curvature of $g_i$ in $\Omega_{i}$ diverges to $\infty$ as $i\to\infty$.
\end{prop}

Proposition \ref{prop-general example} follows easily from Theorem \ref{thrm-large-positive-Ricci}
for the case $(M,g)=(S^n,g_{S^n})$.

We point out that notations in Proposition \ref{prop-general example} is slightly different from
those in Theorem \ref{thrm-large-positive-Ricci}. In Proposition \ref{prop-general example},
$\widetilde {\Omega}_i$ is a domain in $S_n$ and $\Omega_i$ is a domain in $\mathbb R^n$.
We also note that $\Omega_i$ is a bounded domain in $\mathbb R^n$
if the north pole is not in the closure of $\widetilde {\Omega}_i$.

\begin{example}\label{exa-isolated-points}
Let $\{p_1,\cdots,p_k \}$ be a collection of finitely many points in $\mathbb R^n$,
with $k\ge 1$,
and set $\Omega_{R,r}=B_R(0)\backslash \bigcup_{i=1}^{k} B_r(p_i)$.
Then, for $R$ sufficiently large and $r$ sufficiently small,  the complete conformal metric  in $\Omega_{R,r}$
with the constant scalar curvature $-n(n-1)$
has a positive Ricci curvature component somewhere.
Note that the corresponding $\Gamma$ in $S^n$ is
given by $\Gamma=\{e_{n+1}, T(p_1), \cdots, T(p_k)\}$,
which consists of at least two points. If $k=1$ and $p_1=0$, then $\Omega_{R, r}$
is the annular region as in Example \ref{negative ricci domain}.
Combining with  Example \ref{negative ricci domain} for a fixed constant $r$,
we conclude that
the maximum Ricci curvature component of $g_{\Omega_{R,r}}$ tends to $-{n}/{2}$
as $R\to r$ and tends to $\infty$ as $R\to \infty$.
\end{example}


Next, we construct bounded star-shaped domains in which
the complete conformal metrics have positive Ricci components somewhere.

\begin{example}\label{exa-curves}
For $n\ge 4$, set
$$\gamma=\{(0, \cdots, 0, x_n)|\, |x_n|\ge 1\}\subset\mathbb R^n.$$
Let $\Omega_i$ be a sequence of increasing bounded smooth domains in $\mathbb R^n$,
star-shaped with respect to the origin,  which converges to
$\mathbb R^n\setminus \gamma$. Then, for $i$ sufficiently large,
the complete conformal metric in $\Omega_i$
with the constant scalar curvature $-n(n-1)$
has a positive Ricci curvature component somewhere.
Note that the corresponding $\Gamma$ in $S^n$ is
given by the equator in the $x_n$-$x_{n+1}$ plane minus the image under $T$
of the segment $(-1,1)$ on $x_n$-axis. Hence, the dimension of $\Gamma$ is 1.
This is the reason we require $n\ge 4$. We point out that domains in this example are diffeomorphic to balls.
\end{example}


\begin{thebibliography}{DG}

\bibitem{ACF1982CMP} L. Andersson, P. Chru\'sciel, H. Friedrich,
\emph{On the regularity of solutions to the Yamabe equation and the existence of
smooth hyperboloidal initial data for Einstein field equations},
Comm. Math. Phys., 149(1992), 587-612.


\bibitem{A1982CPDE} P. Aviles,
\emph{A study of the singularities of solutions of a class of nonlinear elliptic partial equations},
Comm. Partial Differential Equations,  7(1982), 609-643.

\bibitem{AM1988JDG} P. Aviles, R. C. McOwen,
\emph{Conformal deformation to constant negative scalar curvature on
noncompact Riemannian manifolds},
Journal of Differential Geometry,  27(1988), 225-239.


\bibitem{AM1988DUKE} P. Aviles, R. C. McOwen,
\emph{Complete conformal metrics with negative scalar curvature in compact Riemannian manifolds},
Duke Math. J., 56(1988), 395-398.



\bibitem{Caffarelli1989} L. Caffarelli, B. Gidas, J. Spruck,
\emph{Asymptotic symmetry and local behavior of semi-linear elliptic equations with critical Sobolev growth},
Comm. Pure Appl. Math.,  42(1989), 271-297.

\bibitem{Caffarelli2007} L. Caffarelli, P. Guan, X. Ma,
\emph{A constant rank theorem for solutions of fully nonlinear elliptic equations}. Comm. Pure Appl. Math., 60(2007),  1769-1791.



\bibitem{GaoYau} L. Gao, S.-T. Yau,
\emph{The existence of negatively Ricci curved metrics on three-manifolds},
Invent. Math. 85(1986),  637-652.

\bibitem{Guan2008} B. Guan,
\emph{Complete conformal metrics of negative Ricci curvature on compact manifolds with boundary},
Int. Math. Res. Not., (2008), Article ID rnn105, 25 pp.

\bibitem{M.Gursky1}  M. Gursky, J. Streets, M. Warren,
\emph{Existence of complete conformal metrics of negative Ricci
curvature on manifolds with boundary}, Cal. Var. \& P. D. E., 41(2011), 21-43.

\bibitem{Kennington1985} A. Kennington,
\emph{Power concavity and boundary value problems}, Indiana Univ. Math. J.,
34(1985),  687-704.


\bibitem{Kichenassamy2005JFA}
S. Kichenassamy, \emph{Boundary behavior in the Loewner-Nirenberg problem},
J. of Funct. Anal., 222(2005), 98-113.

\bibitem{Korevaar1999}
N. Korevaar, R. Mazzeo, F. Pacard, R. Schoen,
\emph{Refined asymptotics for constant scalar curvature metrics with isolated singularities},
Invent. Math., 135(1999), 233-272.


\bibitem{LeeParker1987} J. Lee, T. Parker,
\emph{\it The Yamabe problem},
Bull. Amer. Math. Soc. (N.S.),  17(1987), 
37-91.

\bibitem{LiZhu1999}Y. Li, M. Zhu,
\emph{\it Yamabe type equations on three-dimensional Riemannian manifolds},
Comm. Contemp. Math., 1(1999), 
1-50.

\bibitem{Loewner&Nirenberg1974} C. Loewner, L. Nirenberg,
\emph{Partial differential equations invariant under conformal or projective
transformations},
Contributions to Analysis, 245-272, Academic Press, New York, 1974.

\bibitem{Lohkamp5} J. Lohkamp,
\emph{Metrics of negative Ricci curvature},
Ann. of Math., 140(1994),  655-683.




\bibitem{Lohkamp4} J. Lohkamp,
\emph{The higher dimensional positive mass theorem II},
arXiv:1612.07505.

\bibitem{Mazzeo1991} R. Mazzeo,
\emph{Regularity for the singular Yamabe problem}, Indiana Univ. Math. Journal, 40(1991), 1277-1299.


\bibitem{Mazzeo Pacard1999}
R. Mazzeo, F. Pacard, \emph{Constant scalar curvature metrics with isolated singularities}, Duke Math.
J., 99(1999), 353-418.

\bibitem{Mazzeo PollackUhlenbeck1996}
R. Mazzeo, D. Pollack, K. Uhlenbeck, \emph{Moduli spaces of singular Yamabe metrics}, J. Amer. Math.
Soc., 9(1996), 303-344.


\bibitem{Schoen1988}
R. Schoen,
\emph{The existence of weak solutions with prescribed singular behavior for a conformally
invariant scalar equation}, Comm. Pure Appl. Math., 41(1988), 317-392.

\bibitem{SchoenYau1979} R. Schoen, S.-T. Yau, \emph{On the proof of the positive mass conjecture in
general relativity}, Comm. Math. Phys., 65(1979),  45-76.

\bibitem{SchoenYau1988}
R. Schoen, S.-T. Yau, \emph{Conformally
flat manifolds, Kleinian groups and scalar curvature}, Invent.
Math., 92(1988), 47-71.

\bibitem{SchoenYau1994}
R. Schoen, S.-T. Yau, \emph{Lectures on Differential Geometry},
International Press, Cambridge, MA, 1994.

\bibitem{SchoenYau2017} R. Schoen, S.-T. Yau, \emph{Positive scalar curvature and minimal
hypersurface singularities}, arXiv:1704.05490.


\bibitem{Witten1981} E. Witten,  \emph{A new proof of the positive energy theorem}, Comm. Math. Phys.,
80(1981),  381-402.

\end{thebibliography}
\end{document}